\RequirePackage{amsmath}
\documentclass[final]{svjour3}

\usepackage[utf8]{inputenc}

\usepackage{float}
\usepackage{amssymb}
\usepackage{amsmath}
\usepackage{graphicx}
\usepackage[mathscr]{euscript}
\usepackage{lineno}
\usepackage{ulem}
\usepackage{multicol}
\usepackage{lmodern}
\usepackage{lipsum}
\usepackage{marvosym}
\usepackage{mathtools}
\usepackage{caption}
\usepackage{subcaption}
\usepackage{soul}
\setstcolor{red}
\usepackage{fonts}

\newtheorem{prop}{Proposition}
\newtheorem{defn}{Definition}

\newcommand{\dx}{{\Delta x}}
\newcommand{\dy}{{\Delta y}}
\newcommand{\xc}{{x^{\rm{c}}}}
\newcommand{\yc}{{y^{\rm{c}}}}
\newcommand{\argmin}{\operatornamewithlimits{argmin}}
\newcommand{\EMD}{\operatornamewithlimits{EMD}}
\newcommand{\contrast}{\operatornamewithlimits{cont}}

\newcommand{\diag}{\operatornamewithlimits{diag}}
\newcommand{\TV}{\operatornamewithlimits{TV}}

\newcommand{\SNR}{\operatornamewithlimits{SNR}}
\newcommand{\struc}[1] {\operatornamewithlimits{struc}\left [ #1 \right ]}

\newcommand{\Rea}{\mathbb{R}}
\newcommand{\Nat}{\mathbb{N}}
\newcommand{\norm}[1]{\left\lVert#1\right\rVert}

\newcommand{\expec}[1] {\mathbb{E}\left [ #1 \right ] }
\newcommand{\var}{\mathrm{Var}}

\renewcommand{\eqref}[1]{{Eq.~\ref{#1}}}

\newcommand{\defeq}{\vcentcolon=}
\DeclareRobustCommand{\rchi}{{\mathpalette\irchi\relax}}
\newcommand{\irchi}[2]{\raisebox{\depth}{$#1\chi$}}
\DeclareMathOperator{\Tr}{Tr}

\DeclareMathOperator*{\essinf}{ess\,inf}
\DeclareMathOperator*{\nullspace}{null}

\newcommand\restr[2]{{
  \left.\kern-\nulldelimiterspace 
  #1 
  \vphantom{\big|} 
  \right|_{#2} 
  }}

\usepackage{xargs}                      
\usepackage[pdftex,dvipsnames]{xcolor}  
\usepackage[colorinlistoftodos,prependcaption,textsize=tiny]{todonotes}
\newcommandx{\unsure}[2][1=]{\todo[linecolor=red,backgroundcolor=red!25,bordercolor=red,#1]{#2}}
\newcommandx{\change}[2][1=]{\todo[linecolor=blue,backgroundcolor=blue!25,bordercolor=blue,#1]{#2}}
\newcommandx{\info}[2][1=]{\todo[linecolor=OliveGreen,backgroundcolor=OliveGreen!25,bordercolor=OliveGreen,#1]{#2}}
\newcommandx{\improvement}[2][1=]{\todo[linecolor=Plum,backgroundcolor=Plum!25,bordercolor=Plum,#1]{#2}}

\title{Diagnosing Forward Operator Error Using Optimal Transport
\thanks
{This manuscript has been authored, in part, by UT-Battelle, LLC, under Contract No. DE-AC0500OR22725 with the U.S. Department of Energy. The United States Government retains and the publisher, by accepting the article for publication, acknowledges that the United States Government retains a non-exclusive, paid-up, irrevocable, world-wide license to publish or reproduce the published form of this manuscript, or allow others to do so, for the United States Government purposes. The Department of Energy will provide public access to these results of federally sponsored research in accordance with the DOE Public Access Plan (\texttt{http://energy.gov/downloads/doe-public-access-plan}).}}

\titlerunning{Operator Error and Optimal Transport}

\author{Michael A. Puthawala,
Cory D. Hauck,
and Stanley J. Osher
}
\date{\today}

\institute{
	Micheal A. Puthawala \at
	Department of Mathematics
	University of California 
	Los Angeles, CA 90095 \\
	\email{mputhawala@ucla.edu} \\
	This author's research was sponsored by Department of Energy grant DOE-SC0013838 and NSF (STROBE), NSFC 11671005.
	\and
	Cory D. Hauck \at
	Computational Mathematics Group,
	Computer Science and Mathematics Division,
	Oak Ridge National Laboratory,
	Oak Ridge, TN 37831 USA\\
	\email{hauckc@ornl.gov}\\
	This author's research was sponsored by the Office of Advanced Scientific
	Computing Research and performed at the Oak Ridge National Laboratory,
	which is managed by UT-Battelle, LLC under Contract No. DE-AC05-00OR22725.
	\and
	Stanley J. Osher \at
	Department of Mathematics
	University of California 
	Los Angeles, CA 90095 \\
	\email{sjo@math.ucla.edu} \\
	This author's research was sponsored by Department of Energy grant DOE-SC0013838 and NSF (STROBE), NSFC 11671005.
}

\begin{document}

\maketitle

 

\section{Abstract}\label{sec::abstract}

We investigate overdetermined linear inverse problems for which the forward operator may not be given accurately. We introduce a new tool called the {\it structure}, based on the Wasserstein distance, and propose the use of this to diagnose and remedy forward operator error. Computing the structure turns out to use an easy calculation for a Euclidean homogeneous degree one distance, the Earth Mover's Distance, based on recently developed algorithms. The structure is proven to distinguish between noise and signals in the residual and gives a plan to help recover the true direct operator in some interesting cases. We expect to use this technique not only to diagnose the error, but also to correct it, which we do in some simple cases presented below.

\section{Introduction} \label{sec::introduction}

\subsection{Motivation} \label{sec::motivation}

From medical imaging \cite{arridge1999optical} to petroleum engineering \cite{oliver2008inverse} to meteorology \cite{chahine1970inverse}, inverse problems are ubiquitous in science, engineering and mathematics. The goal of such problems is to recover an unknown quantity $u$ given a known forward operator $L$ and measurement $b$ such that $L(u) = b$.  In this work we consider the case where $L$ is a linear operator and write $L(u) \equiv Lu$.  While this choice facilitates a simple analysis in some places, the computational techniques developed here can be extended to consider non-linear operators.

A considerable amount of work has been dedicated to solving inverse problems for a variety of forward operators, especially when $L$ is linear. Powerful techniques have been developed that perform well in the presence of noise in $b,$ singularities in $L$ and various constraints on the solution $u$ \cite{kirsch2011introduction}.

Despite some great successes in the field of inverse problems, there are still mathematical challenges that are difficult to address. One of these, which is important in a bevy of applications, is the calibration of forward operators.  For example, computed tomography (CT) machines are calibrated using known phantoms  for which the desired reconstruction is known exactly \cite{schneider1996calibration}; in synthetic aperture radar, reflectors provide a known ground truth on which devices and reconstruction algorithms are tuned \cite{freeman1992sar}; and in some plasma imaging problems, the forward model has unknown parameters, and the model itself is possibly incomplete \cite{wingen2015regularization}.

Often the calibration problem can be formulated mathematically by considering a family of forward operators $L_\theta,$ parameterized by  $\theta \in \Theta \subset \Rea^p$, with a unique $\hat \theta$ such that $L_{\hat \theta}$ best represents the underlying physical system.  In other words, there exists a $\hat \theta$ such that $L = L_{\hat \theta}$ \cite{schneider2012tomographic,wingen2015regularization}.  If $\hat{\theta}$ is estimated poorly, then an accurate approximation of $u$ is often impossible, even with very sophisticated inverse procedures. 

The problem of detecting forward operator error is similar to that of blind deconvolution in image processing \cite{chan2005image}, where the task is to identify a blurring kernel and recover an image from a given blurry signal. The application of the blurring operator with the image can also be represented in the form $Lu = b$ where the action of $L$ gives the convolution with the blurring kernel. One important difference between the calibration problem considered here and the problem of blind deconvolution is that we will be considering overdetermined problems.

\subsection{Prior Work} \label{sec::prior-work}

Methods for detecting and correcting for errors within the forward operator exist. One approach is total least squares \cite{golub1999tikhonov}, which generalizes the standard least squares method by allowing for error in $L$.  This is expressed by the minimization problem

\begin{align}
    \label{eqn::total-least-squares-problem}
    \begin{alignedat}{2}
        \min_{ \bv, \bJ} &\norm{\bL - \bJ}_F^2 + \norm{\bb - \bJ \bv}^2_2,
    \end{alignedat}
\end{align}
where $\bL$ is the matrix representations of $L$, $\bb$ is the vector representation of $b$, and $\| \cdot \|_F$ is the Frobenius norm.

This approach has the advantage of being relatively easy to analyze, robust under noise in the entries of $\bL$ and solvable using standard linear algebra software. However, for calibration problems, the goal is not to remove entry-wise error in $\bL_\theta$. Instead we seek a value of $\theta \approx \hat \theta$. Total least squares provides good reconstructions when $\bL$ is a matrix whose entries are corrupted by noise. However it requires modification in order to be applied to the parametric calibration problem. In particular, adding the requirement $J = L_{\theta}$ for $\theta \in \Theta$ to Eq. \ref{eqn::total-least-squares-problem} make the resulting minimization problem more difficult to solve, and so may require code beyond standard linear algebra software.
 
Another common approach for calibration is based on Bayesian techniques \cite{kennedy2001bayesian}.  In this setting measured data (possibly noisy) is assumed to be the sum of model output and a discrepancy function, both of which are modeled as Gaussian processes.  We do not go into details of the Bayesian approach in this paper but intend to make comparisons with the EMD approach in future work. However, it is worth noting that the results in this paper do not rely on a Gaussian noise model.
 
Our work is motivated in part by \cite{engquist2013application,engquist2016optimal,yang2018application}, where the authors use the quadratic Wasserstein metric to solve Full-Waveform Inversion (FWI) problems.  In particular, it is demonstrated that the quadratic Wasserstein metric, as opposed to the $L_2$ norm, provides an effective measure of the misfit between given data and computed solution.

\subsection{Our contribution} \label{sec::our-contribution}

In this paper we introduce a new tool, called the structure, that is based on the Earth Mover's Distance (EMD) from optimal transport.  We show that the structure is sensitive to modeling errors in $L$, but insensitive to noise in $b$.
For simple functional forms of $L_\theta$, we demonstrate that the structure can successfully recover the correct parameter $\hat \theta$.  
The method can be implemented as a wrapper around existing inverse problem solvers and thus can be easily integrated into preexisting work flows for solving inverse problems with minimal modifications to existing code bases.  Moreover, due to recent advancements in the calculation of the EMD \cite{li2016fast,li2017parallel}, the additional cost is reasonable.

Our work extends that of \cite{engquist2013application,engquist2016optimal,yang2018application} by considering different inverse problems, a more general noise model, and we use a different Wasserstein metric. See section \ref{sec::comparison-with-prior-work} for more detail. We also show that new algorithms for computing the EMD can be combined with inverse problem solvers to diagnose forward operator error in general inverse problems.

\section{Background} \label{sec::background}

\subsection{Inverse Problems} \label{sec::background::i-p-background}

Let $\cU \subset{L^\infty(X)}$ and $\cB \subset{L^\infty(Y)}$ be function spaces defined over bounded rectangular domains $X \subset \bbR^{d_x}$ and $Y \subset \bbR^{d_y}$, respectively.  We consider problems which come from the discretization of the linear equation
\begin{equation}
    \label{eqn::linear-continuous}
    \cL f = g
\end{equation}
where $f \in \cU$,  $g \in \cB$, and $\cL: \cU \to \cB$ is a bounded linear operator.

To discretize Eq. \ref{eqn::linear-continuous}, we assume that for some $\dx > 0$ and $\dy >0$, $X$ and $Y$ can be partitioned into hypercubes $K^x$ and $K^y$, respectively, of size = $\dx^{d_y}$ and $\dy^{d_y}$, respectively, such that $X = \cup_i \overline{K^x_i}$ and $Y = \cup_j \overline{K^y_j}$. We then let 
\begin{align}
    \cU_{\dx} &= \{ f_{\dx} \in \cU : f_{\dx}|_{K_x}~\text{is constant for all $K_x \subset X$} \} \\
    \cB_{\dy} &= \{ g_{\dy} \in \cB : g_{\dy}|_{K_y}~\text{is constant for all $K_y \subset Y$} \}.
\end{align}
The discrete version of Eq. \ref{eqn::linear-continuous} takes the form
\begin{equation}
    \label{eqn::linear-inverse-problem}
    L u = b  ,
\end{equation}
where $u \in \cU_{\dx}$, $b \in \cB_{\dy}$, and $L \colon \cU_{\dx} \to \cB_{\dy}$ is a bounded linear operator that approximates $\cL$.  The exact forms of $L$, $u$, and $b$ depend on the discretization.  In the appendix, we present a discretization based on the assumption that $\cL$ is generated by line integrals over paths $\cP_y \subset X$ that are parameterized by elements $y \in Y$.

Solving Eq. \ref{eqn::linear-inverse-problem} directly may not be practical if the condition number of $L$ is large, as noise in $b$ can be strongly amplified in the inversion process.  A variational approach to address this difficulty is instead to solve

\begin{equation}
    \label{eqn::linear-regularized-inverse-problem}
    \tilde u = \tilde L^{-1}b 
        \equiv \argmin_{v \in \cU_{\dx}} \norm{Lv - b}_2^2 + \Phi(v;\lambda),
\end{equation}
where $\Phi\colon \cU_{\dx} \rightarrow \Rea^+$ is a regularizing functional with parameter $\lambda \in \Rea^+$. If $\Phi = 0,$ then Eq. \ref{eqn::linear-regularized-inverse-problem} gives the least squares solution of Eq. \ref{eqn::linear-inverse-problem}. Nontrivial examples of $\Phi$ (which may require more regularity than $L^{\infty}(X)$) include 
\begin{enumerate}
    \item $\Phi(v;\lambda) = \lambda\norm{Cv}_2^2$, where the linear operator $C$ approximates a differential operator (Generalized Tikhonov regularization);
    \item $\Phi(v;\lambda) = \lambda\TV(v)$ (Total Variation regularization \cite{rudin1992nonlinear});
    \item $\Phi(v;\lambda) = \lambda\norm{Cv}_1$, where $C$ is a transformation to a space in which $u$ is known to be sparse (Basis Pursuit in Compressed Sensing \cite{goldstein2009split});
    \item a weighted sum of the coefficients in some basis of $U$ (such as a wavelet basis \cite{mallat1989multiresolution,daubechies1988orthonormal} or singular vectors \cite{hansen1993use}).
\end{enumerate}

These regularization methods are able to stably invert the operator $L,$ at least approximately in the sense that $L \tilde u = L\tilde L^{-1}b \approx b$.  Moreover, solutions of Eq. \ref{eqn::linear-regularized-inverse-problem} are able to mitigate the effect of error within $b$; that is, even if $b$ is corrupted (e.g. by noise), $\tilde u$ will be a reasonable reconstruction. In contrast, a modest error in $L$ will likely result in a terrible reconstruction, regardless of the choice of $\Phi$.  An example of this behavior is given in Fig. \ref{fig::bad-reconstruction}.

\begin{figure}
    \centering
    \begin{subfigure}[b]{.28\linewidth}
        \centering
        \includegraphics[width=\textwidth,keepaspectratio]{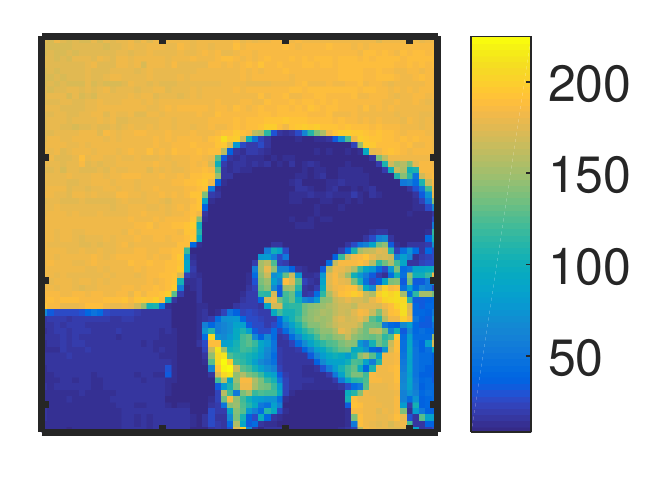}
        \subcaption{Ground truth, $u$.}
        \label{fig::bad-reconstruction::ground-truth}
    \end{subfigure}
    \begin{subfigure}[b]{.28\linewidth}
        \centering
        \includegraphics[width=\textwidth,keepaspectratio]{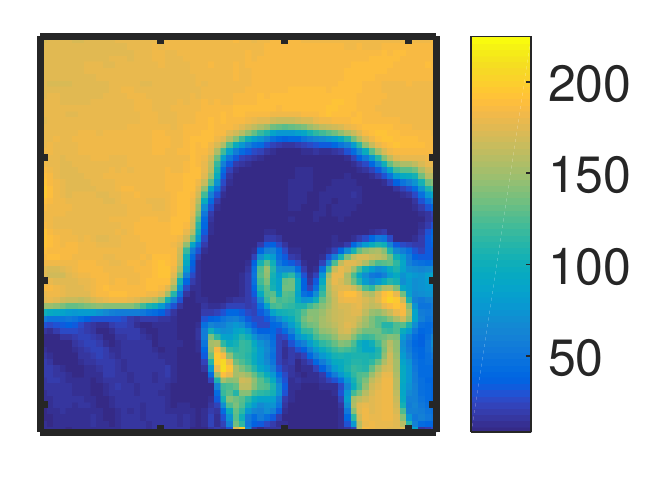}
        \subcaption{$u_{\theta}$ when $\theta = 2.3 = \hat \theta$}
        \label{fig::bad-reconstruction::correct-op}
    \end{subfigure}
    \begin{subfigure}[b]{.28\linewidth}
        \centering
        \includegraphics[width=\textwidth,keepaspectratio]{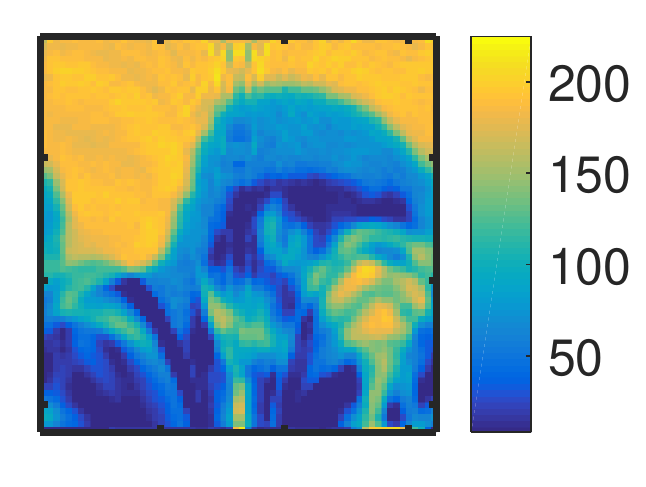}
        \subcaption{$u_{\theta}$ when $\theta = 2.4 \not \approx \hat \theta$}
        \label{fig::bad-reconstruction::incorrect-op}
    \end{subfigure}
    \caption{Demonstration of the sensitivity in the reconstruction in Eq. \ref{eqn::linear-regularized-inverse-problem} to errors in the forward operator.  In this example $L = L_{\hat{\theta}}$ is the `academic operator' from \cite{schneider2012tomographic}, $\theta$ is the parameter $R$ in \cite[Table 1]{schneider2012tomographic}, and $\hat \theta = 2.3$. In this problem Tikhonov regularization was used to define the approximate inverse in Eq. \ref{eqn::linear-regularized-inverse-problem}. }
    \label{fig::bad-reconstruction}
\end{figure}

For the purposes of this paper, we assume that there exists a family $\{ L_\theta\}_{\theta \in \Theta}$ of forward operators parameterized by $\theta \in \Theta$, and a unique $\hat \theta \in \Theta$ such that $L_{\hat \theta}=L$. 
Given a noisy measurement $b + \eta$, where $\eta$ is the noise, and a model parameter $\theta$, the approximate reconstruction of $u$, based on the regularization in Eq. \ref{eqn::linear-regularized-inverse-problem} with operator $L_\theta$, is given by
\begin{equation}
    \label{eqn::linear-noisy-inverse-problem}
    \tilde u_{\theta,\eta} = \tilde L^{-1}_\theta(b + \eta).
\end{equation} 
where the tilde denotes the solution to a regularized problem of the form in Eq. \ref{eqn::linear-regularized-inverse-problem} (where the choice of $\Phi$ is understood).  This notation will be used throughout the remainder of the paper.  

We define the residual as 
\begin{align}
    \label{eqn::inovation-definition}
    r_{\theta,\eta} &= (b + \eta) - L\tilde u_{\theta,\eta} =  (I - L_\theta \tilde L_\theta^{-1} ) (b + \eta)
\end{align}
where $I$ is the identity operator.  The residual is the main object that we study to determine when the parameter $\theta$ is poorly chosen. 

\subsection{Earth Mover's Distance}

A key tool in our analysis of forward operator error is the Earth Mover's Distance.  Below we summarize the presentation in \cite{li2017parallel}.

\begin{defn}[Wasserstein Distance] Let $\Omega \subset \Rea^d$ be convex and compact, and let $c \colon \Omega \times \Omega \rightarrow [0,+\infty)$ be a distance. Given two non-negative distributions $\rho_1 \colon \Omega \rightarrow \Rea^+, \rho_2 \colon \Omega \rightarrow \Rea^+$ such that $\int_\Omega \rho_1 = \int_\Omega \rho_2$. For a given $p \in \Nat$ the $p$'th Wasserstein distance is
\begin{equation}
    \label{eqn::wasserstein-distance-definition}
    \begin{alignedat}{2}
           W_p(\rho_1,\rho_2) = &\left (\min_{\pi \geq 0} \int_{\Omega \times \Omega} c(x^{(1)},x^{(2)})^p\pi(x^{(1)},x^{(2)}) dx^{(1)} dx^{(2)}\right )^{1/p},\\
           &  
            \begin{aligned} 
            \text{\rm subject to:} \quad & \int_\Omega \pi(x^{(1)},x^{(2)})dx^{(2)} = \rho_1(x^{(1)}),\\
                                         & \int_\Omega \pi(x^{(1)},x^{(2)})dx^{(1)} = \rho_2(x^{(2)}).
            \end{aligned}
    \end{alignedat}
\end{equation}
\end{defn}

The function $c$ is called the ground metric and each feasible function $\pi$ is referred to as a transport plan. In this work we set $c(x^{(1)},x^{(2)}) = \norm{x^{(1)} - x^{(2)}}_2$. The Earth Mover's Distance we define here is a special case of the Wasserstein distance where $p = 1$.

\begin{defn}[Earth Mover's Distance] Let $\Omega \subset \Rea^d$ be convex and compact, and let $c \colon \Omega \times \Omega \rightarrow [0,+\infty)$ be a distance. Given two non-negative distributions $\rho_1 \colon \Omega \rightarrow \Rea^+, \rho_2 \colon \Omega \rightarrow \Rea^+$ such that $\int_\Omega \rho_1 = \int_\Omega \rho_2$. The Earth Mover's Distance (EMD) between $\rho_1$ and $\rho_2$ is
\begin{equation}
    \label{eqn::emd-definition}
           \EMD(\rho_1,\rho_2) = W_1(\rho_1,\rho_2).
\end{equation}
\end{defn}
The EMD can also be written in the equivalent form \cite{evans1999differential}

\begin{equation}
    \begin{alignedat}{2}
        \EMD(\rho_1,\rho_2) = & \min_m \int_\Omega\norm{m(x)}_2 dx,\\
        &  
            \begin{aligned}
            \text{\rm subject to:} \quad  & \nabla\cdot m(x) + \rho_2(x) - \rho_1(x) = 0,\\
                                          & m(x) \cdot n(x) = 0 \quad \forall x \in \partial\Omega,
        \end{aligned}
    \end{alignedat} \label{eqn::benier-benamou}
\end{equation}
where $n(x)$ is the normal vector at $x \in \partial\Omega$. This formulation is the basis for recently developed algorithms in \cite{li2016fast,li2017parallel}.

\section{Applying EMD to inverse problems} \label{sec::new-work}

\subsection{Residual and operator correctness}

In a variational reconstruction procedure, the quality of the fit can be investigated by an analysis of $r_{\theta, \eta}$ and $\Phi(\tilde u_{\theta,\eta})$. Generally, the larger $\lambda$ the larger the first term and the smaller the second and vice-versa. Typically the value of $\lambda$ is chosen in an attempt to balance these contributions \cite{hansen1992analysis,hansen1993use}. However if an incorrect forward operator is used, $r_{\theta, \eta}$ will have an additional contribution that does not depend on $\lambda$. 

The characterization above can be made precise in the case of Tikhonov regularization by introducing a matrix notation and using Generalized Singular Value Decomposition \cite[Chapter 8.7.3]{golub1996matrix}. To this end, let $n = \operatorname{dim}(\cU_{\dx}) $ and $m = \operatorname{dim}(\cB_{\dy})$, and expand $u$ and $b$ in terms of characteristic basis functions:
\begin{equation}
    \label{eqn::discrete-basis-expansion}
    u(x) = \sum_{j=1}^n u_j \rchi_{K^x_j}(x) \quad \text{and} \quad b(y) = \sum_{i=1}^m b_i \rchi_{K^y_i}(y).
\end{equation}
Then Eq. \ref{eqn::linear-inverse-problem} becomes
\begin{equation}
    \bL \bu = \bb.
\end{equation}
where $\bu = (u_1,\dots ,u_n)$, $\bb = (b_1,\dots,b_m)$, and $\bL$ has components
\begin{align}
    L_{i,j} = \frac{1}{\Delta y^{d_y}}\int_Y \chi_{K^y_i}L \chi_{K^x_j} dy.
\end{align}

\begin{defn}[GSVD]
Let $\bA \in \Rea^{m \times n}$ and $\bB \in \Rea^{o \times n}$ be two matrices such that $\nullspace(\bA) \cap \nullspace(\bB) = \emptyset$. The Generalized Singular Value Decomposition (GSVD) of the matrix pair $(\bA,\bB)$ is given by
\begin{equation}
   \bA = \bU {\bf\Sigma} \bZ^T 
   \quad \text{and} \quad
   \bB = \bV {\bf\Gamma} \bZ^T  ,
\end{equation}
where $\bU \in \Rea^{m \times n}$ and $\bV \in \Rea^{o \times n}$ are orthogonal; $\bZ \in \Rea^{n \times n}$ is invertible; and 
\begin{align}
    {\bf\Sigma} = \diag(\sigma_1,\dots, \sigma_n)\in \Rea^{n \times n}
    &\quad \text{and} \quad 
    {\bf\Gamma} = \diag(\gamma_1,\dots, \gamma_{n})\in \Rea^{n \times n}
\end{align}
are diagonal matrices such that 
\begin{equation}
    1 \geq \sigma_1 \geq \dots \geq \sigma_n \geq 0
    \quad \text{and} \quad 
    0 \leq \gamma_1 \leq \dots \leq \gamma_n \leq 1,
\end{equation}
with  ${\bf\Sigma}^2 + {\bf\Gamma}^2 = \bI$.
\end{defn}

Using the GSVD, we obtain the following:

\begin{prop}[Residual with Tikhonov regularization] \label{thm::residual-tikhonov}
    Suppose $\bL\bu =\bb$, where $\bL \in \Rea^{m\times n}$ and $m > n$. Let $\tilde \bu_{\theta,\eta}$ be defined by Eq. \ref{eqn::linear-noisy-inverse-problem} with $\Phi(\bv;\lambda) = \lambda\norm{\bC\bv}^2_2$, where $\bC \in \Rea^{o \times n},$ and a noise vector $\bseta \in \bbR^m$ whose elements are independent and spherically symmetric---that is, $\bseta$ and $\bQ\bseta$ have the same probability distribution function for any orthogonal matrix $\bQ \in \Rea^{m \times m}.$
    Assume that $\nullspace(\bL_\theta) \cap \nullspace(\bC) = \emptyset$ so that the GSVD 
    \begin{equation}
    \label{eq::LC-GSVD}
        \bL_\theta = \bU_\theta {\bf\Sigma}_{\theta}\bZ_\theta^T
        \qquad 
        \bC = \bV_\theta {\bf\Gamma}_{\theta}\bZ_\theta^T
    \end{equation}
    for the matrix pair $(\bL_\theta,\bC)$ is well-defined.  Then the residual $\br_{ \theta,\eta}$ associated to $\tilde \bu_{\theta,\eta}$ satisfies the bound
    \begin{align} \label{eqn::residual-tikhonov::incorrect-op}
       \norm{\br_{ \theta, \eta}}^2_2
        \leq&  \norm{(\bI - \bU_\theta\bU_\theta ^T)\bb}^2_2
            +  \norm{(\bb - \bL_\theta \bu)}^2_2 \nonumber\\
            &+ \frac{1}{4} \lambda\norm{\bZ^T_\theta\bu}^2_2 
            + \frac{m - n + \Tr(\hat \bD^2_{\theta,\lambda}))}{m} \expec{\norm{\bseta}^2_2} .
    \end{align}
    The proof of Proposition \ref{thm::residual-tikhonov} is in the appendix.
\end{prop}

This result shows how calibration error can induce $O(1)$ terms (with respect to the regularization parameter $\lambda$) into the residual, the first two terms in Eq.~\ref{eqn::residual-tikhonov::incorrect-op}.  The noise that is orthogonal to the image of $\bL_\theta$ also induces $O(1)$ terms, even if $\theta = \hat \theta$.   Thus it is important to develop tools that can differentiate between these two contributions. For completeness, one should also consider regularization with more general forms of  $\Phi$. Unfortunately in many situations, the operator $\tilde \bL^{-1}_\theta$ is nonlinear, and a rigorous analysis in this vein is much more difficult.

\subsection{Introduction to the structure}

We introduce a mathematical tool to detect contributions to $r_{\theta,\eta}$ that are due to errors in the operator $L$, i.e., when $\theta \neq \hat \theta$, and is insensitive to noise in the residual.  This tool, which we call the structure, is a functional built using the Earth Mover's Distance (EMD). 

\begin{defn}[Structure] \label{def::structure}
For any $f \in L^1(\Omega)$, the structure of $f$ is

\begin{align}
    \struc{f} = \EMD(f^+,f^-),
\end{align}
where
\begin{equation}
    f^+(x) = \max(f(x) - \mu, 0)
   \quad \text{and} \quad
    f^-(x) = \max(\mu - f(x), 0)
\end{equation}
and $\mu = \frac{1}{\norm{\Omega}}\int_\Omega f(x) dx$.
\end{defn}
The following proposition is proven in the appendix.
\begin{prop}[Basic Properties of Structure] \label{thm::EMD-struc-corr-proof}
The operator $\struc{\cdot}$ satisfies the following properties: 
\begin{enumerate}
    \item it is a semi-norm on $L^1(\Omega)$;
    \item for all $g \in L^1(\Omega)$ and $c \in \Rea$,
    \begin{equation}
        \struc{g} = \struc{g + c};
    \end{equation}
    \item $\struc{c} = 0$ for any constant $c \in \bbR$;
    \item if $\rho_1\colon \Omega \rightarrow \Rea^+,$  $\rho_2\colon \Omega \rightarrow \Rea^+$ and $\int_\Omega \rho_1 = \int_{\Omega} \rho_2$,
    \begin{equation}
        \struc{\rho_2 - \rho_1} = \EMD(\rho_1,\rho_2).
    \end{equation}
\end{enumerate}
\end{prop}

Using $\struc{\cdot}$ is a good strategy for detecting operator error for several reasons:

\begin{itemize}
    \item The $\struc{\cdot}$ is small when applied to piecewise noise and large when applied to a (non-constant) smooth function.  (Rigorous statements this effect are made in Section \ref{sec:theorems} below). Thus $\struc{r_{\theta,\eta}}$ will be small when the forward operator is correct and large when it is not.  Although the $\struc{\cdot}$ of a constant is zero, any such contribution to the residual can be discerned by applying a standard norm to its spatial average.

    \item With recent algorithmic advances \cite{li2016fast,li2017parallel}, the underlying $\EMD$ calculation for computing $\struc{\cdot}$ can be performed quickly. For example when $\bb \in \Rea^{256} \times \Rea^{256},$ the structure calculation takes less than a second on consumer grade hardware.
    \item Because its evaluation does not affect the actual inverse procedure, the structure calculation can be incorporated into existing work flows without altering old code. Thus it can be quickly integrated into an existing toolbox for solving inverse problems. 
    \item The $\struc{r_{\theta,\eta}}$ calculation produces not only a number, but also outputs a transport plan (see Figs. \ref{fig::flow-when-alpha-is-0.04}, \ref{fig::flow-when-alpha-is-0.48}). For certain classes of forward operators this additional information can be leveraged to correct forward operators with minimal tuning. This idea will be explored in future work.
\end{itemize}

\subsection{Theoretical Results} \label{sec:theorems}
In this section we establish some theoretical results which support the use of the structure as a tool for diagnosing structural errors in the forward operator of an inverse problem.  The proofs of Theorems \ref{thm::noise-calc}--\ref{thm::smooth-func} are given in Appendix. \ref{sec::proof-of-claims::thm-1-proof}.

   \begin{theorem}[Characterization of noise by structure]\label{thm::noise-calc} Given non-negative integers integers $d$ and $\ell$, let $\Omega = [0,1)^d$ and let $\mathcal{O}_\ell = \left\{ \omega_{\ell,1}, \dots, \omega_{\ell,2^{\ell d}}\right\}$ partition $\Omega$ into $2^{\ell d}$ hypercubes of volume $2^{-\ell d}$. Define $h_\ell: \Omega \to \bbR$ by
    \begin{equation}
        \label{eqn::dyadic-noise-def}
        h_{\ell}(y) = \eta_{\ell,1} \chi_{\ell,1}(y) + \dots + \eta_{\ell,2^{\ell d}}\chi_{\ell,2^{\ell d}}(y)
    \end{equation}
    where
    \begin{equation} 
        \chi_{\ell,i}(y) = \begin{cases}
        1, & x \in \omega_{\ell,i}, \\
        0, & x \not \in \omega_{\ell,i},
        \end{cases}
    \end{equation}
    and $\{\eta_{\ell,i}\}_{i=1}^{2^{\ell d}}$ is a set i.i.d. random variables with mean $\mu$ and variance $\sigma^2$ (See Fig. \ref{fig::noise-split-plot} for a visualization of $h_\ell$.) If $\epsilon_\ell = 2^{-\ell}$, then as $\ell \rightarrow \infty$, $ \epsilon_\ell \rightarrow 0$ and
    \begin{equation}
        \expec{\struc{h_\ell}} \leq \sigma\begin{cases}
    - \epsilon_\ell\log{\epsilon_\ell}, & d = 2,\\
    2\sqrt{d}\epsilon_\ell, & d > 2,
    \end{cases}
    \end{equation}
    \end{theorem}
    where the expectation is with respect to the weights $\eta_{\ell,i}$.

     \begin{figure}[ht!]
    \centering
    \begin{subfigure}{.24\linewidth}
        \centering
        \includegraphics[width=\textwidth,keepaspectratio]{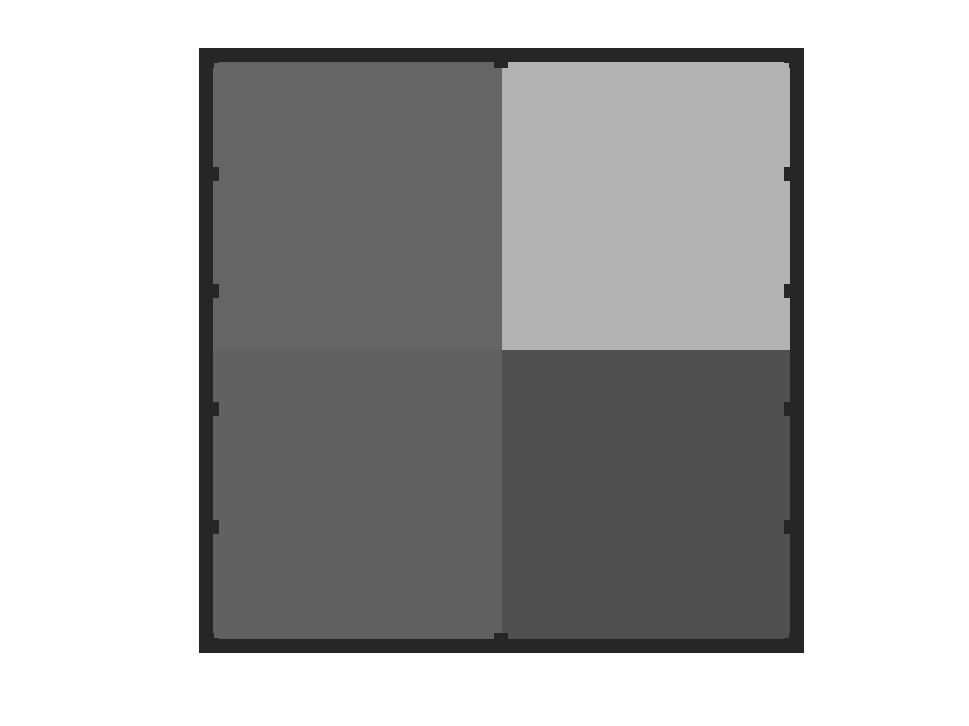}
        \subcaption{$h_1$}
        \label{fig::noise-split-plot-n-1}
    \end{subfigure}
    \begin{subfigure}{.24\linewidth}
        \centering
        \includegraphics[width=\textwidth,keepaspectratio]{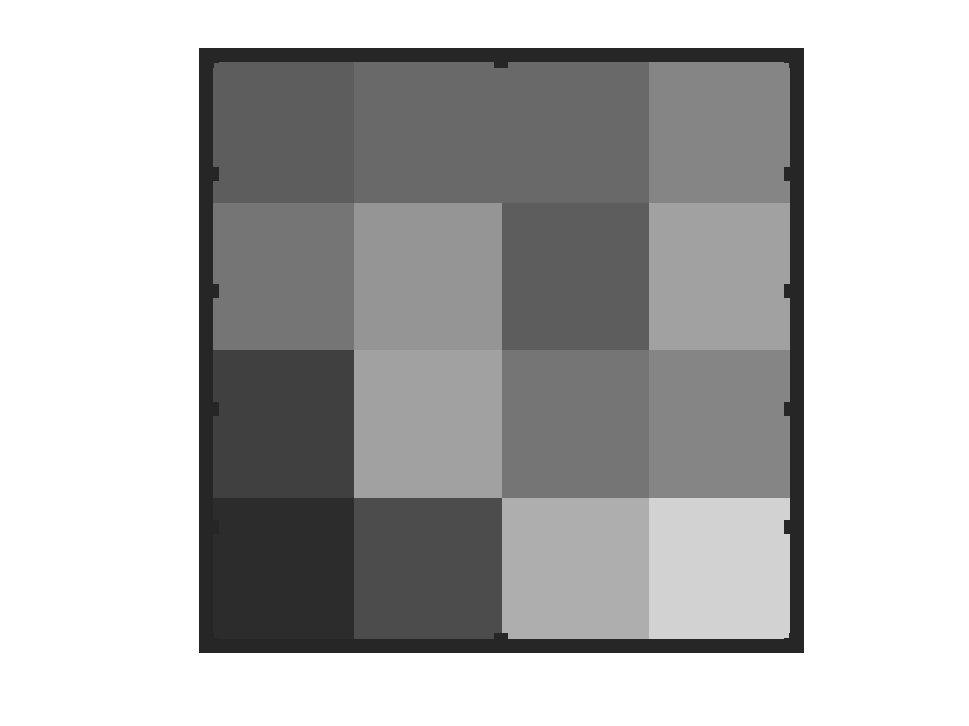}
        \subcaption{$h_2$}
        \label{fig::noise-split-plot-n-2}
    \end{subfigure}
    \begin{subfigure}{.24\linewidth}
        \centering
        \includegraphics[width=\textwidth,keepaspectratio]{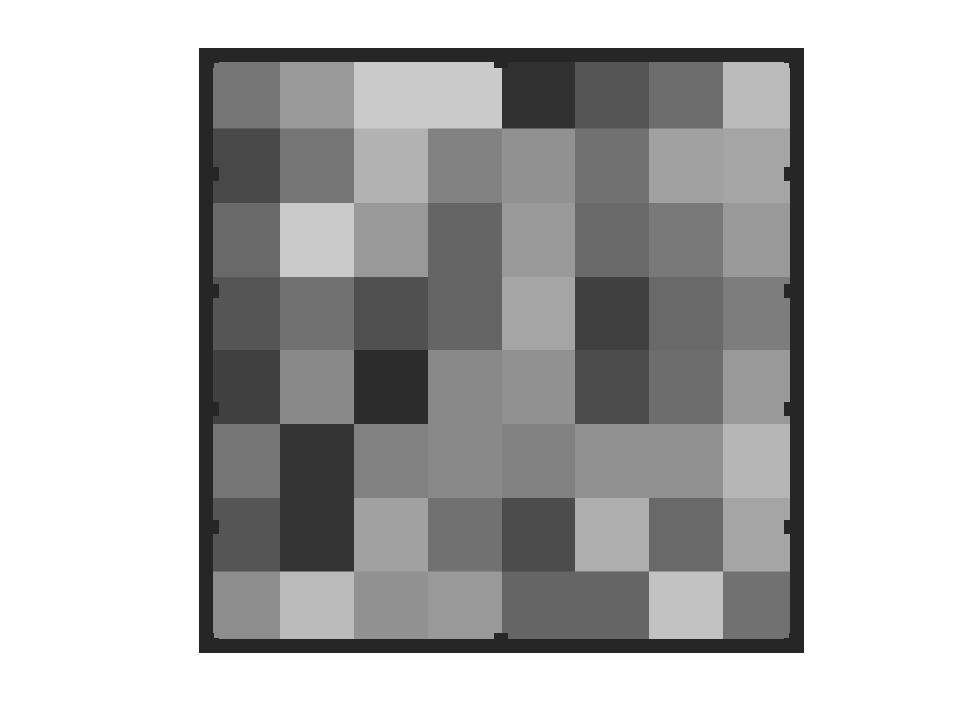}
        \subcaption{$h_3$}
        \label{fig::noise-split-plot-n-3}
    \end{subfigure}
    \begin{subfigure}{.24\linewidth}
        \centering
        \includegraphics[width=\textwidth,keepaspectratio]{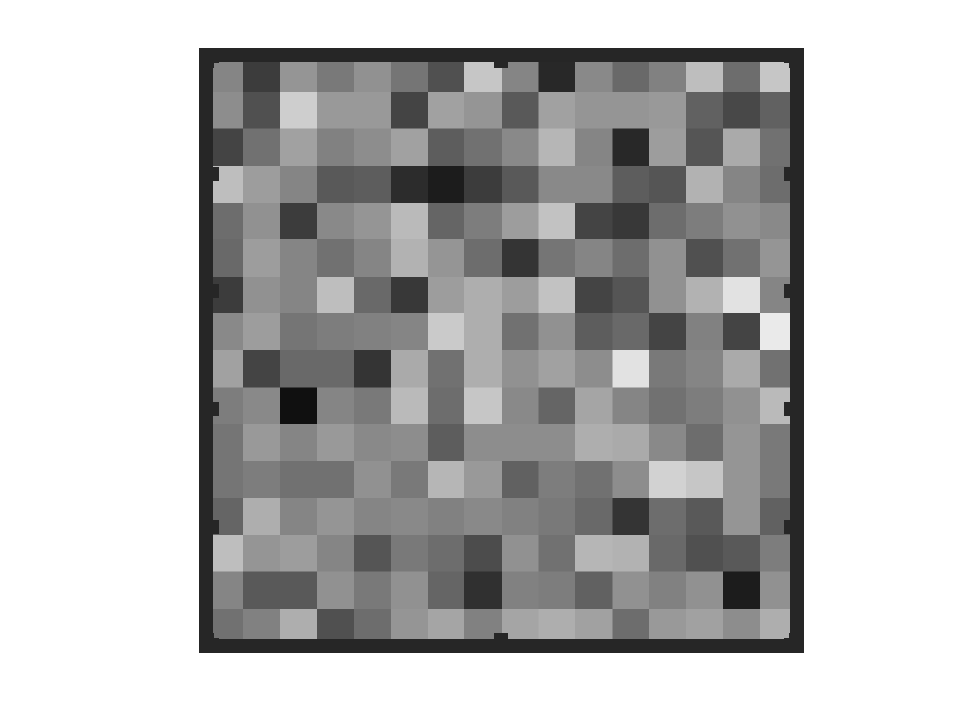}
        \subcaption{$h_4$}
        \label{fig::noise-split-plot-n-4}
    \end{subfigure}
    \caption{Example of $h_\ell$ when $d = 2$, $\mu = 0,$ and $\sigma = 1$.} 
    \label{fig::noise-split-plot}
\end{figure}

    \begin{lemma}[L2 norm of Noise]\label{thm::noise-calc-l2}
    Given the assumptions of Thm. \ref{thm::noise-calc}, suppose further that $\mu =0$.  Then
    \begin{equation}
        \expec{\norm{h_\ell}_2^2} = \sigma^2,
    \end{equation}
    where the expectation is with respect to the weights $\eta_{\ell,i}.$
    \end{lemma}
    
    \begin{theorem}[Characterization of a smooth function by structure]\label{thm::smooth-func}  Given the assumptions of Thm. \ref{thm::noise-calc}, let $R_{\ell} \colon \mathcal{B} \rightarrow \mathcal{B}_{\epsilon_\ell}$. If 
    \begin{align}
        R_{\ell} \phi (y) = \frac{1}{\omega_{\ell,i}}\int_{\omega_{\ell,i}} \phi(z) dz, \quad \forall y \in \omega_{\ell,i}.
    \end{align}
    where $\phi \in C^1\left (\overline Y \right)$ then
    \begin{equation}
        \left |\struc{R_\ell \phi} - \struc{\phi}\right | \leq C(|\nabla \phi|) \, d\epsilon_\ell^{2},
    \end{equation}
    where the constant $C$ depends on the maximum of $\nabla\phi$ on $\overline Y$. In particular,
    \begin{equation}
         \struc{R_\ell \phi} \rightarrow \struc{\phi} \mbox{ as } \ell \rightarrow +\infty.
    \end{equation}
    \end{theorem}

\subsection{Comparison with prior work} \label{sec::comparison-with-prior-work}

The work here is inspired, in part, by the study of seismic imaging inverse problems in \cite{engquist2013application,engquist2016optimal,yang2018application}. There the authors measure the misfit between simulated and measured data using the Wasserstein distance squared $W_2^2(\rho_1,\rho_2) = \left (W_2(\rho_1,\rho_2) \right)^2$. To handle the possibly negative distributions, the authors in \cite{engquist2013application,engquist2016optimal,yang2018application} introduce the {\it misfit} function
\begin{align}
    \label{eqn::engquist-distance-definition}
    d(f,g) &= W^2_2\left (\frac{\max(f,0)}{\int \max(f,0)dx}, \frac{\max(g,0)}{\int \max(g,0)dx}\right ) \nonumber\\
    & \quad + W^2_2\left (\frac{\max(-f,0)}{\int \max(-f,0)dx}, \frac{\max(-g,0)}{\int \max(-g,0)dx}\right )
\end{align}
which plays a similar role to $\struc{f - g}$ in this work. In \cite[Section 2.6]{engquist2013application} the authors show that $d$ is insensitive to noise, with a scaling result that is similar to Thm. \ref{thm::noise-calc} up to a logarithmic factor. Specifically, if $f$ and $g$ are two non-negative functions such that $f - g$ has the form of $h_\ell$, defined in Eq. \ref{eqn::dyadic-noise-def}), with uniformly distributed noise, then
\begin{equation}
    \label{eqn::engquist-noise-result}
    d(f,g) = O(\epsilon_\ell).
\end{equation}

The approach taken in \cite{engquist2013application,engquist2016optimal,yang2018application} differs from the approach in this paper in at least two key ways. First is the choice of $W^2_2$ rather than $W_1$.  This has the following consequences:

\begin{itemize}
\item  $W_2$ and $W_2^2$ have the property of \textit{cyclic monotonicity} (see \cite[Sec. 2.1]{evans1997partial} for a definition and proof), which can be used to show convexity of $d$ with respect to shifts, dilation and partial amplitude loss. In this work we make no such claims about the convexity of $\struc{\cdot}$.

\item As a semi-norm, the EMD (like all $W_p$ for $p \in [1,\infty)$) is a degree-one homogeneous functional and satisfies a triangle inequality  (see \cite[p. 94]{villani2008optimal}.  The functional $W^2_2$ has neither property. For example of the latter, let $f = 2\chi_{0,1/2}$, $h = 2\chi_{1/2,1}$ and $g = 2\chi_{1,3/2}$. Then $W^2_2(f,h) = \frac{1}{4}$,  $W^2_2(h,g) = \frac{1}{4}$ but $W^2_2(f,g) = 1$, then
\begin{equation}
    \label{eqn::engquist-triangle-violation}
    W^2_2(f,g) > W^2_2(f,h) + W^2_2(h,g).
\end{equation}
 
\item Redefining $d$ with $W_2$ instead of $W^2_2$ would recover a triangle inequality and degree-one homogeneity.  However, the cost of such a modification would be to increase the sensitivity of $d$ to noise.  Indeed, the scaling in Eq. \ref{eqn::engquist-noise-result} would change from $O(\epsilon_\ell)$ to $O(\epsilon_\ell^{1/2})$, which is significantly slower than the scaling in Thm. \ref{thm::noise-calc}.  

\item Finally, $W_1$ is more  directly analogous to the definition of work used throughout physics, distance times effort. Consider the case when 
\begin{align}
    f(x) = \frac{1}{2}\chi_{[0,2]}(x) \quad g(x) = \frac{1}{2}\chi_{[1,3]}(x)
\end{align}
and the two transport plans
\begin{align}
    \label{eqn::example-1-wasserstein-optimal-transport-plan}
    \pi_1(x_1, x_2) &= \begin{cases} 1/2 \text{ if } x_2 = 1 + x_1  \text{ and } x_1 \in [0,2]\\
    0 \text{ otherwise}
    \end{cases}\\
    \pi_2(x_1, x_2) &= \begin{cases} 1/2 \text{ if } x_2 = 2 + x_1  \text{ and } x_1 \in [0,1]\\
    0 \text{ otherwise }
    \end{cases}
\end{align}
The cost of $\pi_1$ as measured by $W_2$ is twice that of $\pi_2$. Both plans cost the same as measured by $W_1$. In words $W_2$ `prefers' to make many smaller movements as opposed to fewer larger movements, while $W_1$ is agnostic to such differences.
\end{itemize}

The second key difference between the approach in \cite{engquist2013application,engquist2016optimal,yang2018application} and the approach taken here lies in the  definition of $d$ and $\struc{\cdot}$, both of which are used to address the fact that the Wasserstein metric is only defined for non-negative distributions with the same mass. It is worth noting that $d(f,g)$ and $\struc{\cdot}$ could be defined using any Wassterstein metric.  However, $d$ introduces several undesirable artifacts.

\begin{itemize}
 \item The normalization in the definition means that
\begin{align}
   d(\lambda f, \lambda g) = d(f,g), \quad \forall \lambda \in \Rea^+ .
\end{align}
In particular, unlike $\struc{\cdot}$, it is not degree-one homogeneous.

\item Special care is required in the case that $\max(f,0) \equiv 0$ but $\max(g,0) \not \equiv 0$. Indeed one of the reasons that the results in Eq. \ref{eqn::engquist-noise-result} require $f$ and $g$ to be positive and differ only by uniform noise is that small changes is the noise can alter the support of $\max(f,0)$ and $\max(g,0)$. The $\struc{\cdot}$ has no such restrictions on the noise model.

\item The $\struc{\cdot}$ is continuous w.r.t. the $L_1(\Omega)$ norm provided that $\Omega$ is bounded (see Lemma \ref{lem::EMD-bound-by-L1-norm}). $d(f,g)$, however, is not. For example consider, the functions
\begin{align}
    f_\epsilon = \chi_{[\epsilon,1 - \epsilon]} - \epsilon \chi_{(1 - \epsilon,1]}, \quad
    g_\epsilon = - \epsilon \chi_{[0,\epsilon)} + \chi_{[\epsilon,1 - \epsilon]} - \epsilon \chi_{(1 - \epsilon,1]}.
\end{align}
Clearly $f_\epsilon - g_\epsilon \to 0$ in $L_1(\Omega)$ as $\epsilon \to 0$; however, 
\begin{align}
    \lim_{\epsilon \rightarrow 0} d(f_\epsilon,g_\epsilon) \geq \lim_{\epsilon \rightarrow 0} \frac{1}{2}\left (1 + \frac{\epsilon}{4}\right)^2 = \frac{1}{2}.
\end{align}
This lack of continuity due to sign changes is one of the reasons for having restrictions on the noise model for $d(f,g)$. 

\item The kernel of $\struc{\cdot}$ consists of constant functions, and so $\struc{f - g} = 0 \iff f = g + c$ for some constant $c$. This $c$ is easily recovered  by computing the difference between the averages if $f$ and $g$.On the other hand, the kernel of $d$ is
\begin{equation}
\operatorname{Ker}(d) = \left\{
\begin{array}{ll}
 (f,g) \in L^1 \times L^1 : & \max(f,0) = \lambda_+ \max(g,0)\text{ and }  \\
& \max(-f,0) = \lambda_- \max(-g,0) \quad \text{for }\lambda_+, \lambda_- \in \Rea^+
\end{array}
\right\}
\end{equation}
It is more difficult to account for such a kernel.
\end{itemize}

\section{Numerical Results} \label{sec::numerical-results}

In this section we present the results of several numerical experiments.  We make two simplying assumptions.  First, we let $X$ and $Y$ be two dimensional domains. This choice is motivated by ease of visualization as well as the availability of code to quickly compute the $\EMD$ in two dimensions. We, however, believe that our results generalize well to high dimensional problems.  Second, we assume that $L_\theta$ is linear in $\theta$.   This choice is for simplicity, but it also is a reasonable approximation for finding a local optimum.  Indeed, if $L_\theta$ smoothly depends on $\theta$, then $L$ is locally linear:
\begin{align}
\label{eqn::linearized-forward-op}
L_{\hat{\theta} + \delta\theta} = L_{\hat{\theta}} + \nabla_\theta L({\hat{\theta}})\cdot \delta\theta  + O(\delta\theta^2).
\end{align}

For each experiment, we provide with a known signal $u$ and a family of operators $\{L_\theta\}_{\theta \in \Theta}$.  
We then set $L = L_{\hat \theta}$ for some $\hat \theta \in \Theta$, generate a measurement $b = L_{\hat \theta}u$, and examine the behavior of $\struc{r_{\theta,\eta}}$ as a function of $\theta$. The expectation is that 
\begin{equation}
    \hat \theta \approx \theta^* \defeq \argmin_{\theta\in \Theta} \struc{r_{\theta,\eta}} .
\end{equation}

The first two experiments show that indeed $\theta^* \approx \hat \theta$ even with relatively high noise. The final experiment illustrates that the method performs better as the problem becomes more overdetermined. We report a figure of merit, the contrast, defined as:
\begin{equation}
    \label{eqn::contrast-definition}
    \contrast(F) = \frac{\max(F) - \min(F)}{\max(F) + \min(F)}
\end{equation}
for any $F \colon \Theta \rightarrow \Rea^+$ that is not identically zero. 
The contrast measures the depth of a minimum, and the greater the contrast, the less the location of the minimum changes in the presence of additive noise in $F$. In all three experiments we compare the contrast of $\struc{\cdot}$ with the discrete norms $\norm{\cdot}_1$ and $\norm{\cdot}_2$.  For any  $z \in \mathcal{B}_{\Delta y}$ these norms are given by, 
\begin{align}
    \norm{z}_1 = \Delta y^2 \sum_{i_1,i_2} |z_{i_1,i_2}|
    \quad \text{and} \quad
    \norm{z}_2 = \Delta y \left (\sum_{i_1,i_2} z_{i_1,i_2}^2\right )^{1/2}
\end{align}
We also generate plots of all three (semi-) norms as a function of the parameter $\theta$.

\subsection{Implementation Details}

The implementation of each of these experiments involves four basic steps: (i) the generation of the random forward operators $L_\theta$; (ii) generation of the signal $u$, measurement $b$ and noise $\eta$; (iii) calculation of $\tilde u_{\theta,\eta}$; and (iv) computation of the $\struc{\cdot}$. The specific values of parameters needed to recreate our results are given in Table \ref{table::expriment-parameters}.

\begin{table}
    \centering
    \begin{tabular}{l l|l l l|l l l}
         Parameter & Value & Parameter & Value & Ref. & Parameter & Value & Ref.\\\hline\hline
         \multicolumn{2}{l}{\textbf{Discretization}\footnotemark}& \multicolumn{3}{l}{\textbf{Inversion}}&\multicolumn{3}{l}{$\struc{\cdot}$}\\\hline
         $\Delta x$ & 1/64& $\Phi(\cdot,\lambda)$  &$\lambda \TV(u)$   & \cite{rudin1992nonlinear}& Max Iter       & 8000  & \cite{li2017parallel}\\
         $\Delta y$& 1/100 & $\lambda$              & 10                & \cite{rudin1992nonlinear}& $\EMD_\mu$     & 7e-6  & \cite{li2017parallel}\\
         && $\mu$                  & 100               & \cite{goldstein2009split}& $\EMD_\tau$    & 3     & \cite{li2017parallel}\\
         && Bregman Iterations     & 10                & \cite{goldstein2009split}&                &       &\\
    \end{tabular}
    \caption{Numerical parameters for Experiments 1 - 3.}
    \label{table::expriment-parameters}
\end{table}
\footnotetext{$\Delta x$ and $\Delta y$ both change for Experiment 3, however the other parameters are fixed.}
\begin{enumerate}
    \item {\bf Generation of the random forward operators.} Recall the definitions in Section \ref{sec::background::i-p-background}. A forward operator $L_\theta$, even an academic one, but rather a the discretization of an operator $\mathcal{L} \colon \mathcal{U} \rightarrow \mathcal{B}.$ In applications, $L_\theta$ models the action of some physical process which produces a measurement. For example in seismic imaging the forward operator is the propagation of a seismic wave \cite{engquist2013application}, and in plasma imaging in tokamaks the forward operator couples the optics of the camera with the symmetries of the plasma \cite{wingen2015regularization}.
    
    For our experiments, we presume that $\mathcal{L}$ is a Line Integral Operator (LIO). (See Appendix \ref{sec::LIO} for details.) If $f \colon X \rightarrow \Rea$ and $g \colon Y \rightarrow \Rea,$ then for each $y \in Y,$ $g(y)$ represents the integral of $f$ over some path $p(y).$ Some examples of common LIO are the Radon, Abel and Helical Abel transforms \cite{schneider2012tomographic}.
    
    \item {\bf Generation of the signal, measurement and noise.} The underlying signal $u \in \mathcal{U}_{\Delta x}$ is a series of concentric rings (see Fig. \ref{fig::exp-1::ground-truth}). Then we apply $L_{\hat\theta}$ to $u$ to obtain a noiseless measurement $b \in \mathcal{B}_{\Delta y}$ (see Fig. \ref{fig::exp-1::true-emission}). The noisy signal (see Fig. \ref{fig::exp-1::given-emission}) is generated by adding independent white noise $\eta$ with mean zero and variance $\sigma$ to each element of $b$ so that
    \begin{equation}
        \label{eqn::SNR-definition}
        \SNR = \frac{\norm{b}_2}{\norm{\eta}_2}
    \end{equation}
    is at a specified level.
    
    \begin{figure}[htbp]
        \centering
        \begin{subfigure}[b]{.31\linewidth}
            \centering
            \includegraphics[width=\textwidth,keepaspectratio]{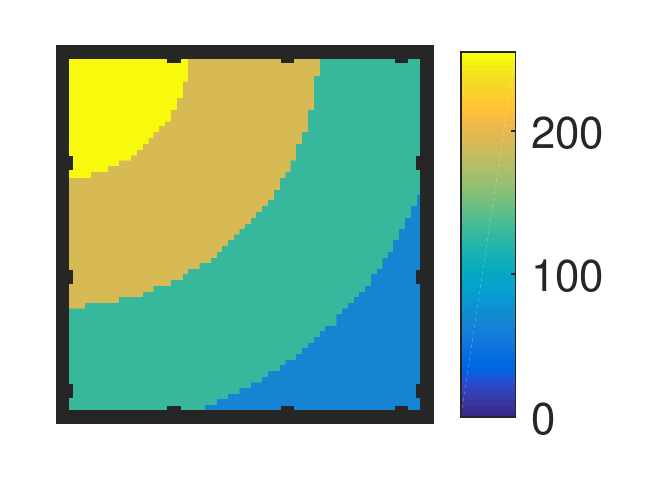}
            \subcaption{$u$.}
            \label{fig::exp-1::ground-truth}
        \end{subfigure}
        \begin{subfigure}[b]{.31\linewidth}
            \centering
            \includegraphics[width=\textwidth,keepaspectratio]{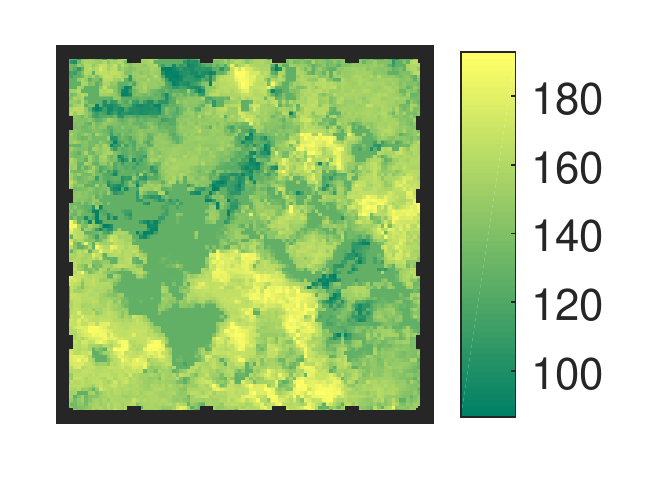}
            \subcaption{$b = L_{\hat \theta}u$.}
            \label{fig::exp-1::true-emission}
        \end{subfigure}
        \begin{subfigure}[b]{.31\linewidth}
            \centering
            \includegraphics[width=\textwidth,keepaspectratio]{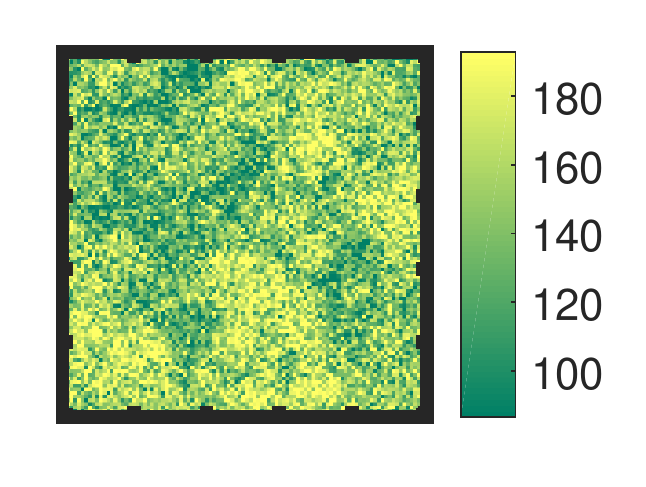}
            \subcaption{$b + \eta$.}
            \label{fig::exp-1::given-emission}
        \end{subfigure}
        \caption{The signal $u$, measurement $b$, and noisy measure $b + \eta$ for Experiment 1.}
        \label{fig::exp-1::measurement}
    \end{figure}
    
    \item {\bf Computation of $\tilde u_{\theta, \eta}$.} Throughout these experiments, we use the inversion procedure of the form of Eq. \ref{eqn::linear-regularized-inverse-problem} with $\Phi(v;\lambda) = \lambda \norm{\bC v}_1$ where $\bC$ is a one-sided discrete approximation of the gradient operator:
    \begin{align}
        (C v)_{2i,j} &= \frac{1}{dx}\left(v_{i,j} - v_{\ell-1,j} \right) \nonumber \\
        (C v)_{2i+1,j} &= \frac{1}{dy}\left(v_{i,j} - v_{i,j-1} \right)
    \end{align}
    where $v_{i,j}$ is the $i$'th x and $j$'th y component of the vector $\bv$, and likewise for $(\bC \bv)_{i,j}$.This is TV regularization and has found wide success within image processing, especially when the underlying signal to be recovered is piecewise constant \cite{goldstein2009split,rudin1992nonlinear}.
    
    To solve the resulting non-linear variational problem, we use the Split-Bregman algorithm, specifically the Generalized Split-Bregman Algorithm (GSBA) of \cite{goldstein2009split}, which requires specification of a step size parameter $\mu$ (called $\lambda$ in \cite{goldstein2009split}). GSBA requires the repeated solution of the linaer system $(\bL^T\bL + \lambda^2 \bC^T\bC)x = y$. The matrix $(\bL^T\bL + \lambda^2 \bC^T\bC)$ is sparse and so we solve it using the L-BFGS \cite{Becker2012lbfgs,zhu1994lbfgs} method (limited memory Broyden-Fletcher-Goldfarb-Shanno\cite{broyden1970convergence,fletcher1970new,goldfarb1970family,shanno1970conditioning}).
    
    \item {\bf Computation of the $\struc{\cdot}$.} Computing $\struc{\cdot}$ requires computing $\EMD$. The algorithm that we use is given in \cite{li2016fast,li2017parallel,ryu2018transport}.
\end{enumerate}

\subsection{Experiment 1}

This experiment is based on a normalized Eq. \ref{eqn::linearized-forward-op} where $p = 1$. Let $L_0$ and $L_1$ be two operators generated as described in Appendix \ref{sec::LIO}. We define $\theta \in [0,1]$ and 
\begin{equation}
    L_\theta = (1 - \theta)L_0 + \theta L_1.
\end{equation}
Fig. \ref{fig::residual-vs-alpha} is a plot of the residual for different values of $\theta$. In Fig. \ref{fig::residual-when-alpha-is-0.04}, $\theta = 0.04$, and in Fig. \ref{fig::residual-when-alpha-is-0.48} $\theta = 0.48$. Upon close inspection, one can see that from Fig. \ref{fig::residual-when-alpha-is-0.04} that when $\theta$ is small the residual visually looks like white noise, whereas from Fig. \ref{fig::residual-when-alpha-is-0.48} when $\theta$ is large the residual has underlying structure in addition to the noise. It is, however, difficult to see. Despite these two plots appearing similar they have very different structures, $\struc{r_{0.04,\eta}} \approx 0.06$ and $\struc{r_{0.48,\eta}} \approx 0.54$. The structure is also evident by looking at Figs. \ref{fig::flow-when-alpha-is-0.04}, \ref{fig::flow-when-alpha-is-0.48}, which are $m$ from Eq. \ref{eqn::benier-benamou}. Note that when $\theta = 0.04,$ $m$ is higgledy-piggledy, whereas when $\theta = 0.48,$ $m$ appears much more orderly.

\begin{figure}[t]
    \centering
    \begin{subfigure}{.25\linewidth}
        \centering
        \includegraphics[width=\textwidth]{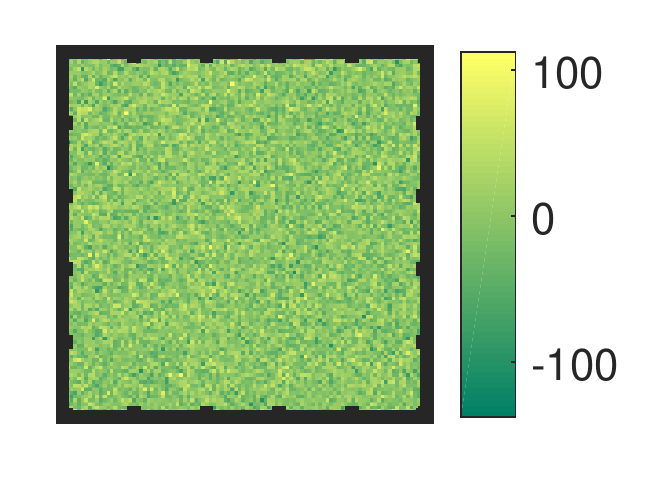}
        \subcaption{$r_{0.04,\eta}$}
        \label{fig::residual-when-alpha-is-0.04}
    \end{subfigure}
    \begin{subfigure}{.23\linewidth}
        \centering
        \includegraphics[width=\textwidth]{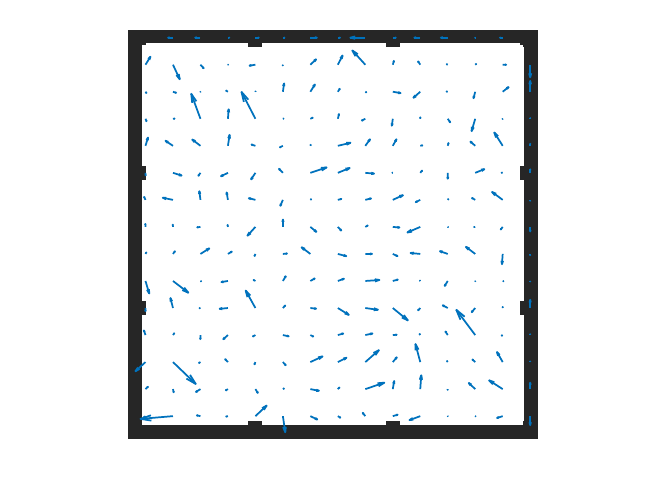}
        \subcaption{$\hat{m}_{0.04}$}
        \label{fig::flow-when-alpha-is-0.04}
    \end{subfigure}
    \begin{subfigure}{.25\linewidth}
        \centering
        \includegraphics[width=\textwidth]{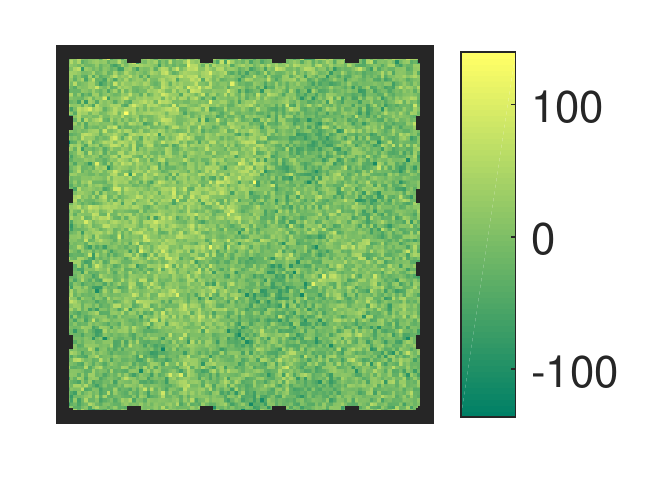}
        \subcaption{$r_{0.48,\eta}$}
        \label{fig::residual-when-alpha-is-0.48}
    \end{subfigure}
    \begin{subfigure}{.23\linewidth}
        \centering
        \includegraphics[width=\textwidth]{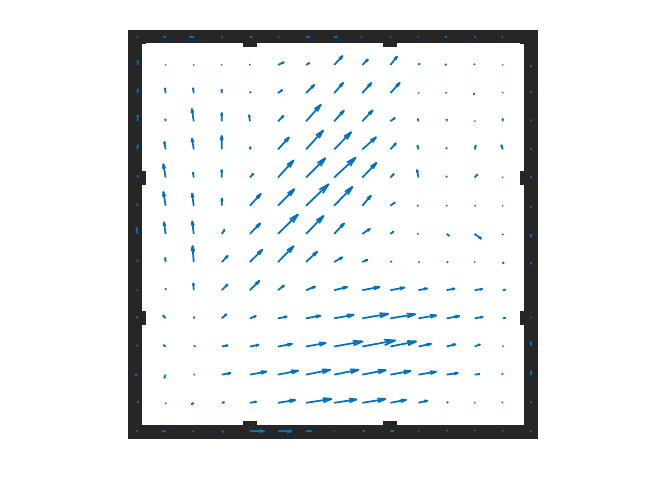}
        \subcaption{$\hat{m}_{0.48}$}
        \label{fig::flow-when-alpha-is-0.48}
    \end{subfigure}
    \caption{Results from Experiment 1. The residual and flow $\hat{m}_\theta$ that minimizes Eq. \ref{eqn::benier-benamou} for a given $\theta$. In Figs. \ref{fig::flow-when-alpha-is-0.04} and \ref{fig::flow-when-alpha-is-0.48}, the orientation of the arrows represents the direction $\hat{m}_\theta$, and the length of the arrows is proportional to the magnitude. 
}
    \label{fig::residual-vs-alpha}
\end{figure}

A plot of $\struc{r_{\theta,\eta}}$ vs $\theta$ is given in Fig. \ref{fig::exp-1}. Clearly, $\struc{r_{\theta,\eta}}$ is minimized when $\theta \approx 0$. Further,  we note that $\struc{r_{\theta,\eta}}$ is increasing as a function of $\theta$ when $\theta \in [0,0.5],$ however then decreases. This is expected behavior around the minimum, however the problem is evidently not convex away from $\hat \theta$. This is important to keep in mind for future work.

\begin{figure}[htbp]
    \centering
    \begin{subfigure}[b]{.31\linewidth}
        \centering
        \includegraphics[width=\textwidth,keepaspectratio]{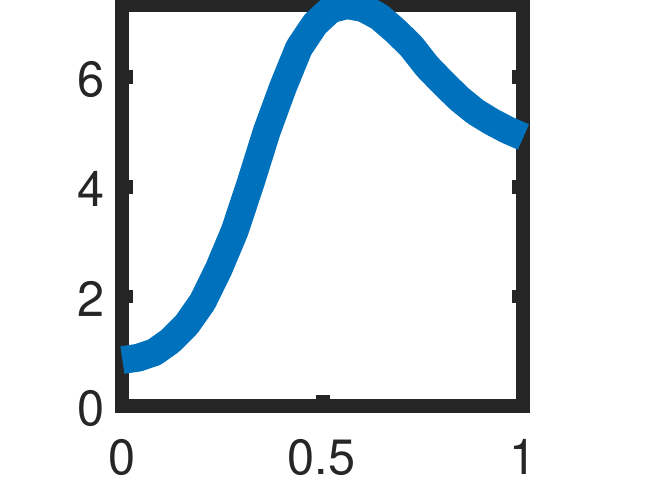}
        \label{fig::exp-1::struc-vs-alpha}
        \subcaption{$\struc{r_{\theta,\eta}}$ vs $\theta$.}
    \end{subfigure}
    \begin{subfigure}[b]{.31\linewidth}
        \centering
        \includegraphics[width=\textwidth,keepaspectratio]{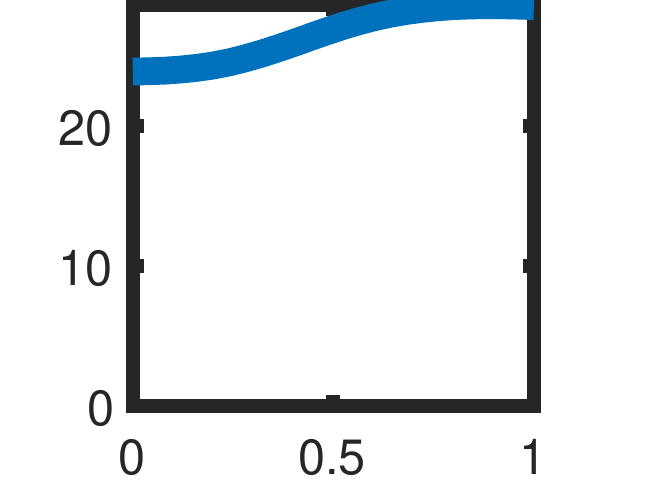}
        \label{fig::exp-1::L-1-vs-alpha}
        \subcaption{$\norm{r_{\theta,\eta}}_1$ vs $\theta$.}
    \end{subfigure}
    \begin{subfigure}[b]{.31\linewidth}
        \centering
        \includegraphics[width=\textwidth,keepaspectratio]{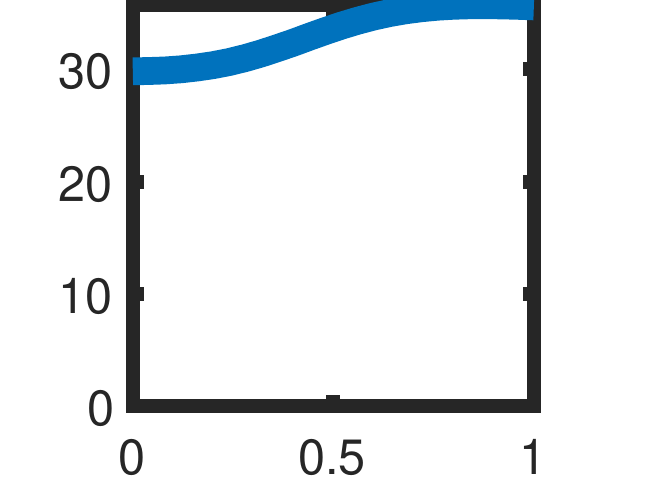}
        \label{fig::exp-1::L-2-vs-alpha}
        \subcaption{$\norm{r_{\theta,\eta}}_2$ vs $\theta$.}
    \end{subfigure}
    \caption{Results from Experiment 1. The value of $r_{\theta,\eta}$ as measured by $\struc{\cdot},$ $\norm{\cdot}_1$ and $\norm{\cdot}_2$. In all examples the minimum occurs when $\theta = 0$ however the contrast is greatest for $\struc{\cdot}$.}
    \label{fig::exp-1}
\end{figure}
\subsection{Experiment 2}
Experiment 2 is also based on a normalized Eq. \ref{eqn::linearized-forward-op}, however in this case $p = 2$ and $\hat \theta = \left (\frac{1}{2}, \frac{1}{2}\right )$.
The true signal used in Experiment 2 is the same as in Experiment 1 (see Fig. \ref{fig::exp-1::ground-truth}). This experiment studies the change in the contrast for $\struc{\cdot}, \norm{\cdot}_1$ and $\norm{\cdot}_2$ as the $\SNR$ decreases. The results are summarized in Table \ref{table::exp-2::results}.

\begin{figure}[ht!]
    \centering
    \begin{subfigure}[b]{.32\linewidth}
        \centering
        \includegraphics[width=\textwidth]{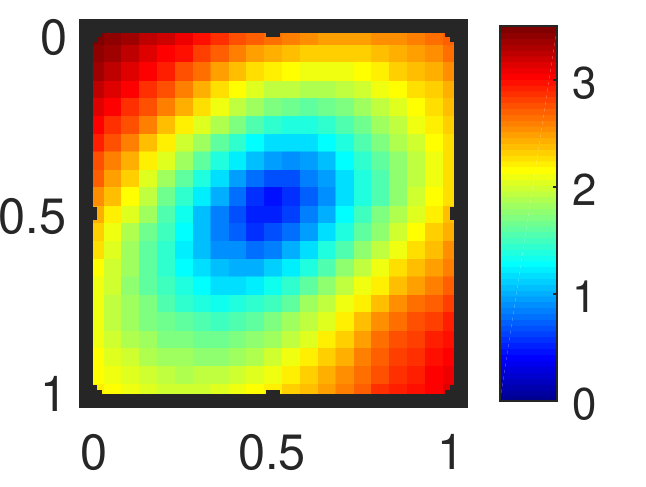}
        \subcaption{$\struc{r_{\theta,\eta}}$ vs $\theta$}
        \label{fig::exp-2::SNR-25::struc}
    \end{subfigure}
    \begin{subfigure}[b]{.32\linewidth}
        \centering
        \includegraphics[width=\textwidth]{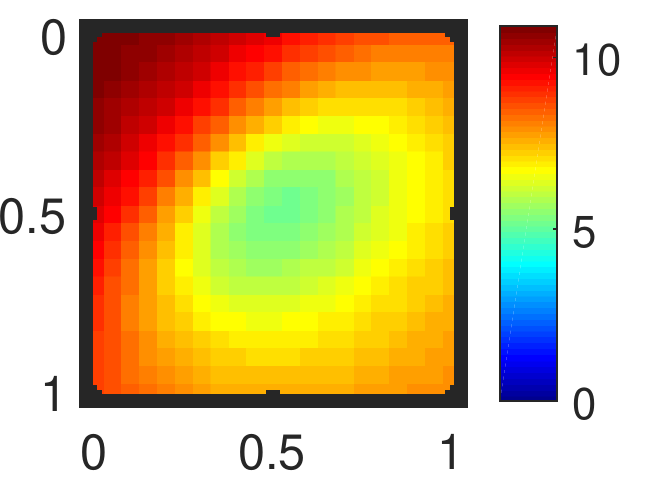}
        \subcaption{$\norm{r_{\theta,\eta}}_1$ vs $\theta$}
        \label{fig::exp-2::SNR-25::L-1}
    \end{subfigure}
    \begin{subfigure}[b]{.32\linewidth}
        \centering
        \includegraphics[width=\textwidth]{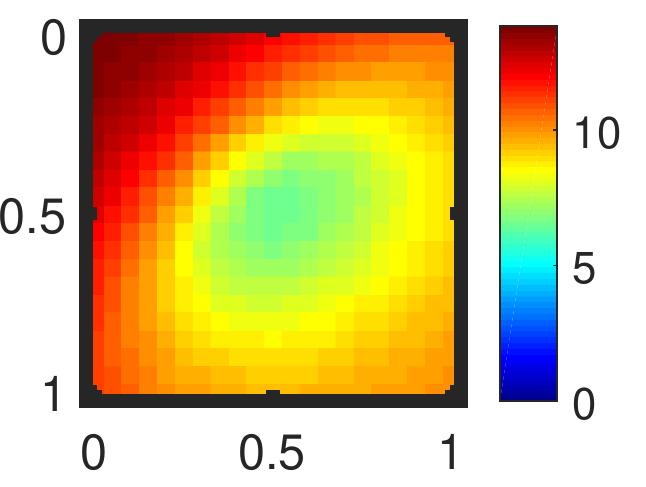}
        \subcaption{$\norm{r_{\theta,\eta}}_2$ vs $\theta$}
        \label{fig::exp-2::SNR-25::L-2}
    \end{subfigure}
    \caption{Results from Experiment 2. In these plots $\SNR = 25$.}
    \label{fig::exp-2::SNR-25}
\end{figure}

\begin{figure}[ht!]
    \centering
    \begin{subfigure}[b]{.31\linewidth}
        \centering
        \includegraphics[width=\textwidth,height=\textheight,keepaspectratio]{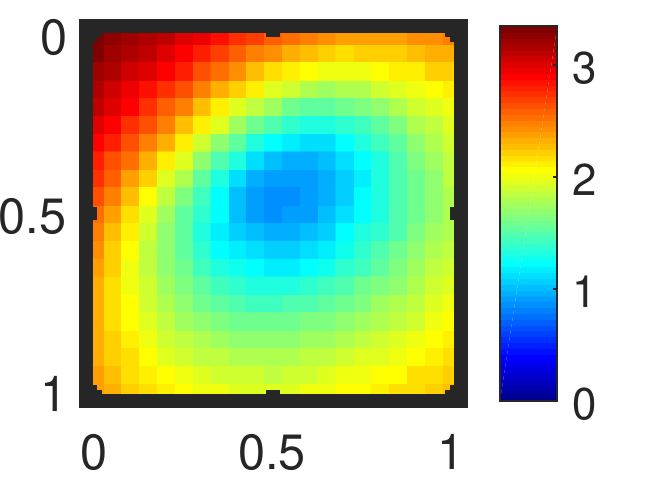}
        \subcaption{$\struc{r_{\theta,\eta}}$ vs $\theta$}
        \label{fig::exp-2::SNR-5::struc}
    \end{subfigure}
    \begin{subfigure}[b]{.31\linewidth}
        \centering
        \includegraphics[width=\textwidth,height=\textheight,keepaspectratio]{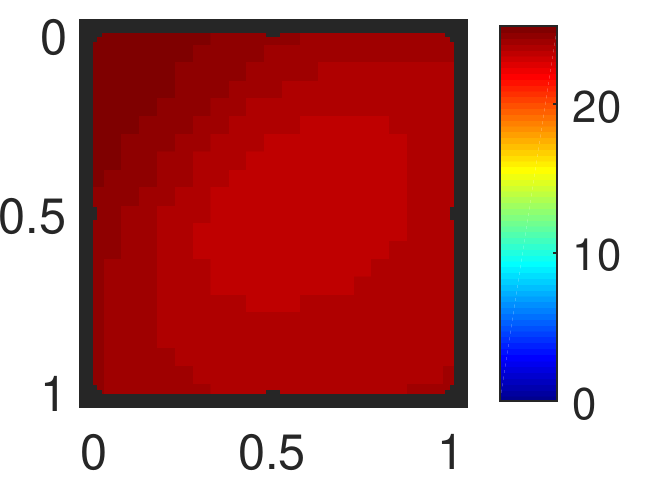}
        \subcaption{$\norm{r_{\theta,\eta}}_1$ vs $\theta$}
        \label{fig::exp-2::SNR-5::L-1}
    \end{subfigure}
    \begin{subfigure}[b]{.31\linewidth}
        \centering
        \includegraphics[width=\textwidth,height=\textheight,keepaspectratio]{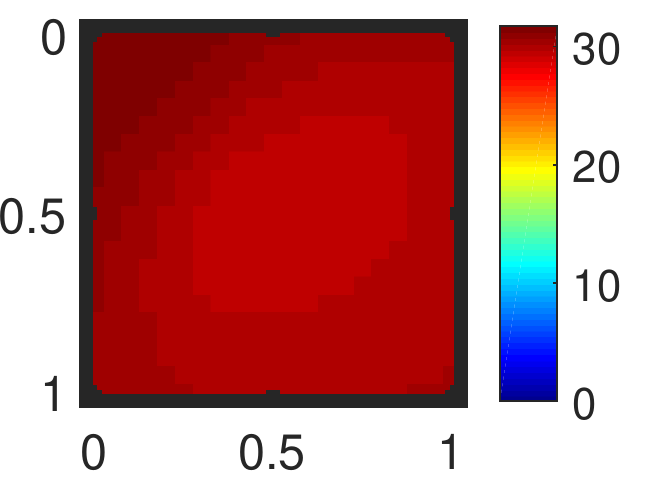}
        \subcaption{$\norm{r_{\theta,\eta}}_2$ vs $\theta$}
        \label{fig::exp-2::SNR-5::L-2}
    \end{subfigure}
    \caption{Results from Experiment 2. In these plots $\SNR = 5$.}
    \label{fig::exp-2::SNR-5}
\end{figure}

\begin{table}[htb]
    \centering
    \begin{tabular}{c||c|c|c}
         Contrast&$\struc{r_{\theta,\eta}}$&$\norm{r_{\theta,\eta}}_1$&$\norm{r_{\theta,\eta}}_2$ \\\hline \hline
         $\SNR = 25$&0.7547&0.3493&0.3544\\ \hline
         $\SNR = 5$ &0.5917&0.0398&0.0404
    \end{tabular}
    \caption{Results from Experiment 2. The contrast for different choices of (semi)norms. Larger is better.}
    \label{table::exp-2::results}
\end{table}

In all cases, the contrast of $\struc{\cdot}$ is greatest, and the contrast of $\struc{\cdot}$ relative to $\norm{\cdot}_1$ of $\norm{\cdot}_2$ increases as the problem becomes more noisy. This suggests that $\struc{\cdot}$ is a more robust choice of semi-norm for measuring the level of miscalibration of $L_\theta$, especially when noise levels are high.

\subsection{Experiment 3}

The final experiment examines the necessity of the overdetermined assumption of $L_\theta$. We repeat the setup of Experiment 2; however we fix the $\SNR = 25$ and instead adjust $\Delta y$ so that $L_\theta\colon \mathcal{U}_{\Delta x} \rightarrow \mathcal{B}_{\Delta y}$ becomes a square operator. We start with a fixed reference $\Delta y_0$, and consider
\begin{align}
    \label{eqn::exp-3::subsampling}
    \mathcal{B}_{\Delta y_0} \cong \Rea^{100\times100} \quad \mathcal{B}_{4/3\Delta y_0} \cong \Rea^{75\times75} \quad \mathcal{B}_{2\Delta y_0} \cong \Rea^{50\times50} \quad \mathcal{B}_{4\Delta y_0} \cong \Rea^{25\times25}.
\end{align}
In all cases, $\mathcal{U}_\Delta x \cong \Rea^{25 \times 25}$ is fixed. Each of the $\mathcal{B}$ in Eq. \ref{eqn::exp-3::subsampling} are plotted in Fig. \ref{fig::exp-3::signals}. The values of
\begin{align}
    \theta^s = \argmin_{\theta \in \Theta} \struc{r_{\theta,\eta}} \quad \theta^1 = \argmin_{\theta \in \Theta} \norm{r_{\theta,\eta}}_1 \quad \theta^2 = \argmin_{\theta \in \Theta} \norm{r_{\theta,\eta}}_2
\end{align}
as well as the contrast are recorded in Table \ref{table::exp-3::results}. Finally, plots of $\struc{r_{\theta,\eta}}$, $\norm{r_{\theta,\eta}}_1$, and $\norm{r_{\theta,\eta}}_2$ vs $\theta$ as $\Delta y$ increases are Figs. \ref{fig::exp-3::100x100}, \ref{fig::exp-3::75x75}, \ref{fig::exp-3::50x50}, and \ref{fig::exp-3::25x25}.

\begin{figure}[ht!]
    \centering
    \begin{subfigure}[b]{.23\linewidth}
        \centering
        \includegraphics[width=\textwidth,height=\textheight,keepaspectratio]{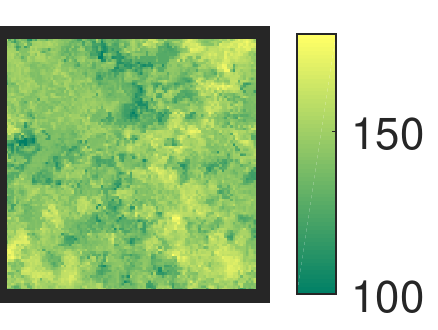}
        \subcaption{$b \in \Rea^{100\times 100}$}
        \label{fig::exp-3::100x100::signal}
    \end{subfigure}
    \begin{subfigure}[b]{.23\linewidth}
        \centering
        \includegraphics[width=\textwidth,height=\textheight,keepaspectratio]{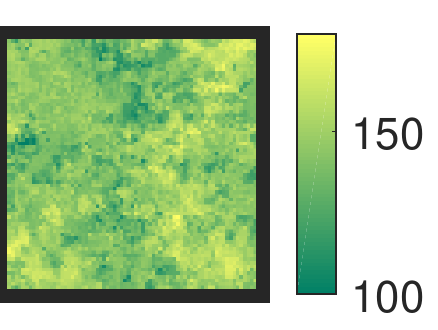}
        \subcaption{$b \in \Rea^{75\times 75}$}
        \label{fig::exp-3::75x75::signal}
    \end{subfigure}
    \begin{subfigure}[b]{.23\linewidth}
        \centering
        \includegraphics[width=\textwidth,height=\textheight,keepaspectratio]{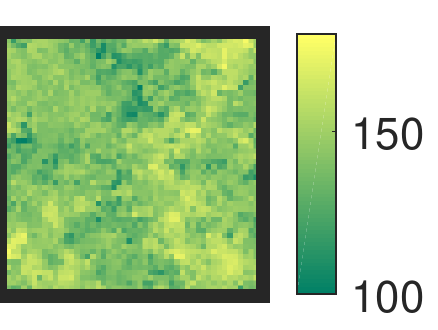}
        \subcaption{$b \in \Rea^{50\times 50}$}
        \label{fig::exp-3::50x50::signal}
    \end{subfigure}
    \begin{subfigure}[b]{.23\linewidth}
        \centering
        \includegraphics[width=\textwidth,height=\textheight,keepaspectratio]{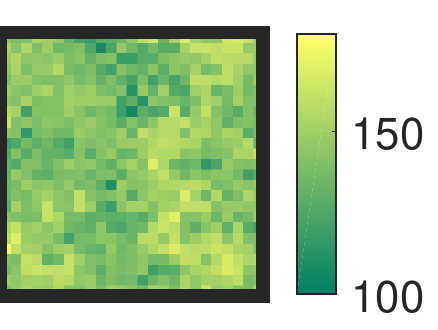}
        \subcaption{$b \in \Rea^{25\times 25}$}
        \label{fig::exp-3::25x25::signal}
    \end{subfigure}
    \caption{Results from Experiment 3. Plot of $b$ for various choices of $\Delta y$ (see Eq. \ref{eqn::exp-3::subsampling}).}
    \label{fig::exp-3::signals}
\end{figure}

\begin{figure}[ht!]
    \centering
    \begin{subfigure}[b]{.31\linewidth}
        \centering
        \includegraphics[width=\textwidth,height=\textheight,keepaspectratio]{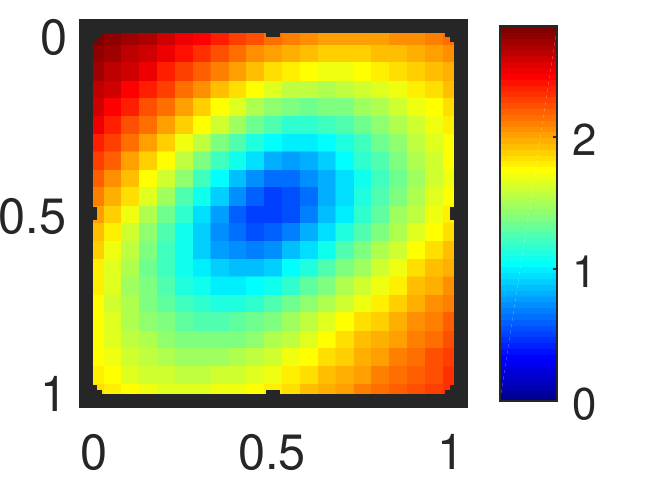}
        \subcaption{$\struc{r_{\theta,\eta}}$ vs $\theta$}
        \label{fig::exp-3::100x100::struc}
    \end{subfigure}
    \begin{subfigure}[b]{.31\linewidth}
        \centering
        \includegraphics[width=\textwidth,height=\textheight,keepaspectratio]{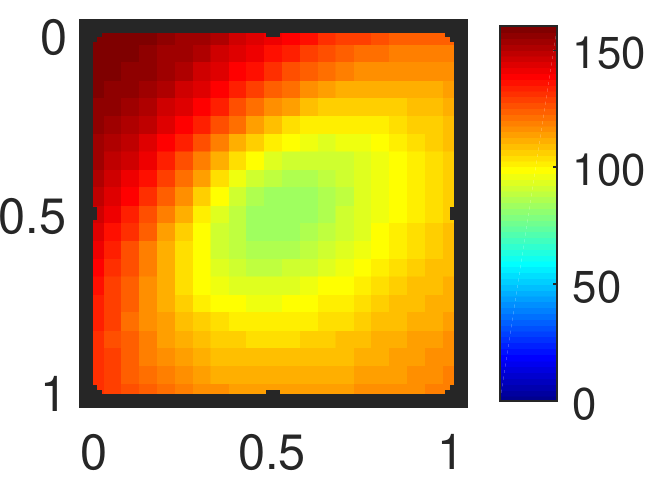}
        \subcaption{$\norm{r_{\theta,\eta}}_1$ vs $\theta$}
        \label{fig::exp-3::100x100::L-1}
    \end{subfigure}
    \begin{subfigure}[b]{.31\linewidth}
        \centering
        \includegraphics[width=\textwidth,height=\textheight,keepaspectratio]{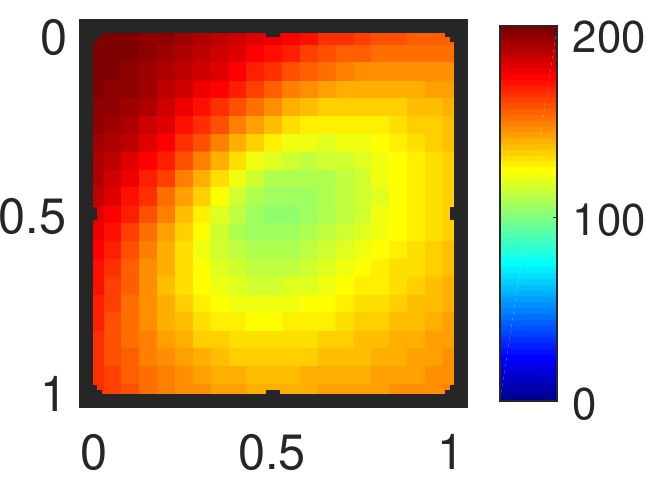}
        \subcaption{$\norm{r_{\theta,\eta}}_2$ vs $\theta$}
        \label{fig::exp-3::100x100::L-2}
    \end{subfigure}
    \caption{Results from Experiment 3. In these plots $L\colon \Rea^{25\times 25} \rightarrow \Rea^{100\times 100}$. See Table \ref{table::exp-3::results} for the contrast, $\theta^s$, $\theta^1$, and $\theta^2$.}
    \label{fig::exp-3::100x100}
\end{figure}

\begin{figure}[ht!]
    \centering
    \begin{subfigure}[b]{.31\linewidth}
        \centering
        \includegraphics[width=\textwidth,height=\textheight,keepaspectratio]{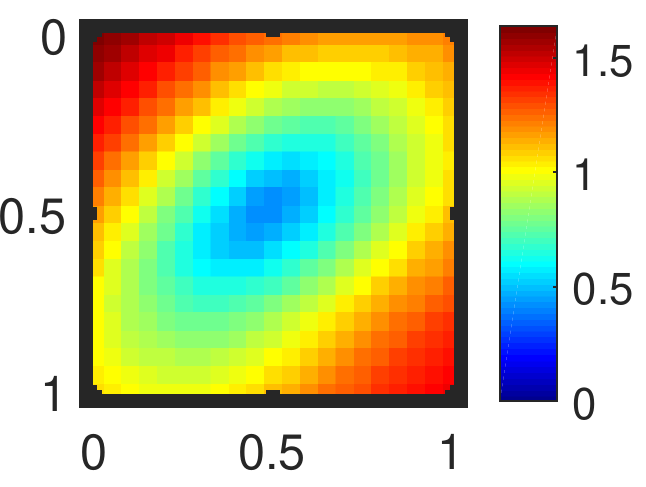}
        \subcaption{$\struc{r_{\theta,\eta}}$ vs $\theta$}
        \label{fig::exp-3::75x75::struc}
    \end{subfigure}
    \begin{subfigure}[b]{.31\linewidth}
        \centering
        \includegraphics[width=\textwidth,height=\textheight,keepaspectratio]{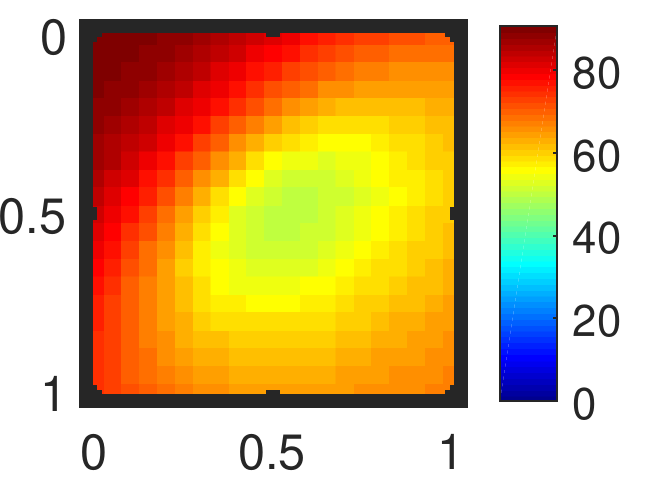}
        \subcaption{$\norm{r_{\theta,\eta}}_1$ vs $\theta$}
        \label{fig::exp-3::75x75::L-1}
    \end{subfigure}
    \begin{subfigure}[b]{.31\linewidth}
        \centering
        \includegraphics[width=\textwidth,height=\textheight,keepaspectratio]{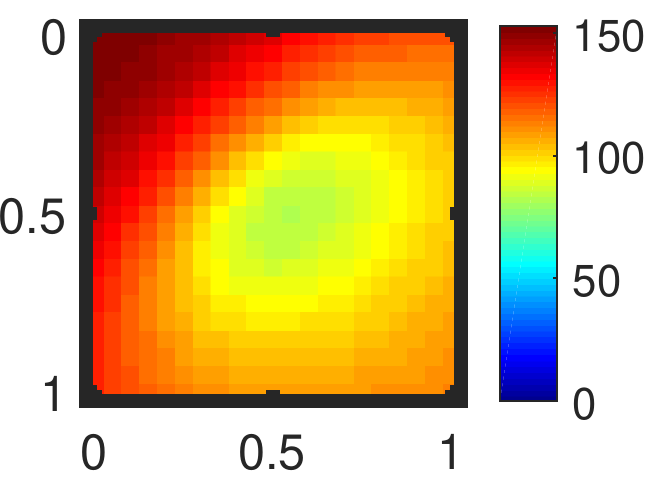}
        \subcaption{$\norm{r_{\theta,\eta}}_2$ vs $\theta$}
            \label{fig::exp-3::75x75::L-2}
    \end{subfigure}
    \caption{Results from Experiment 3. In these plots $L\colon \Rea^{25\times 25} \rightarrow \Rea^{75\times 75}$. See Table \ref{table::exp-3::results} for the contrast, $\theta^s$, $\theta^1$, and $\theta^2$.}
    \label{fig::exp-3::75x75}
\end{figure}

\begin{figure}[ht!]
    \centering
    \begin{subfigure}[b]{.31\linewidth}
        \centering
        \includegraphics[width=\textwidth,height=\textheight,keepaspectratio]{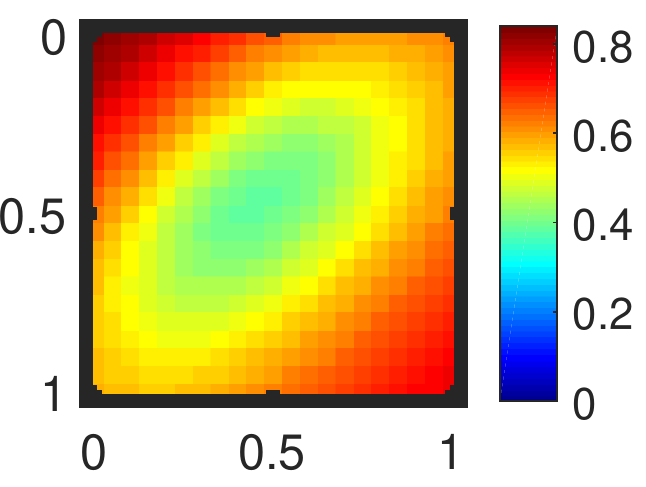}
        \subcaption{$\struc{r_{\theta,\eta}}$ vs $\theta$}
        \label{fig::exp-3::50x50::struc}
    \end{subfigure}
    \begin{subfigure}[b]{.31\linewidth}
        \centering
        \includegraphics[width=\textwidth,height=\textheight,keepaspectratio]{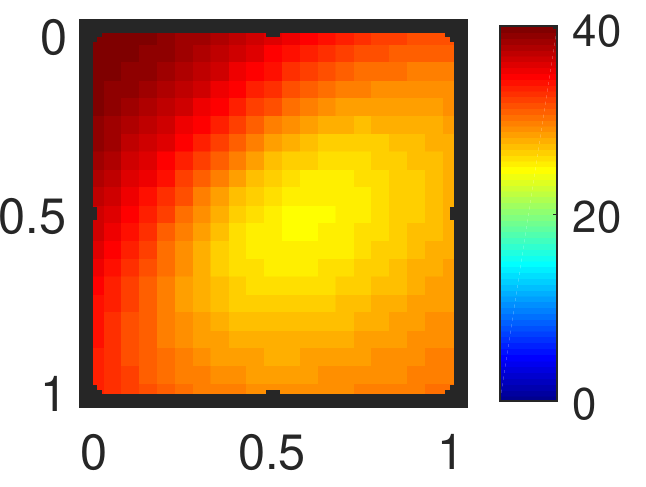}
        \subcaption{$\norm{r_{\theta,\eta}}_1$ vs $\theta$}
        \label{fig::exp-3::50x50::L-1}
    \end{subfigure}
    \begin{subfigure}[b]{.31\linewidth}
        \centering
        \includegraphics[width=\textwidth,height=\textheight,keepaspectratio]{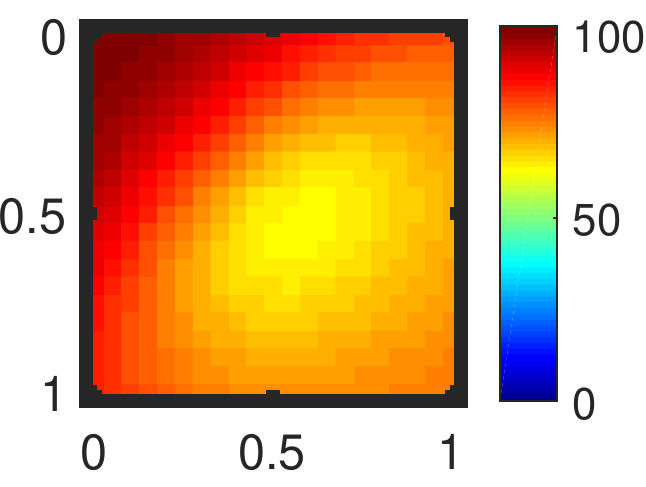}
        \subcaption{$\norm{r_{\theta,\eta}}_2$ vs $\theta$}
        \label{fig::exp-3::50x50::L-2}
    \end{subfigure}
    \caption{Results from Experiment 3. In these plots $L\colon \Rea^{25\times 25} \rightarrow \Rea^{50\times 50}$. See Table \ref{table::exp-3::results} for the contrast, $\theta^s$, $\theta^1$, and $\theta^2$.}
    \label{fig::exp-3::50x50}
\end{figure}

\begin{figure}[ht!]
    \centering
    \begin{subfigure}[b]{.31\linewidth}
        \centering
        \includegraphics[width=\textwidth,height=\textheight,keepaspectratio]{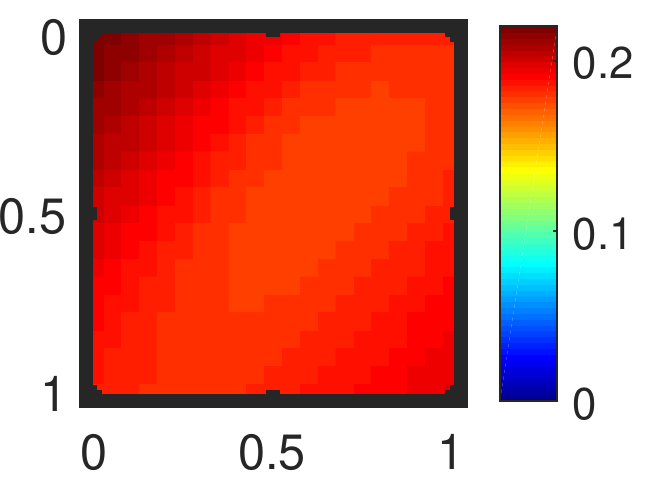}
        \subcaption{$\struc{r_{\theta,\eta}}$ vs $\theta$}
        \label{fig::exp-3::25x25::struc}
    \end{subfigure}
    \begin{subfigure}[b]{.31\linewidth}
        \centering
        \includegraphics[width=\textwidth,height=\textheight,keepaspectratio]{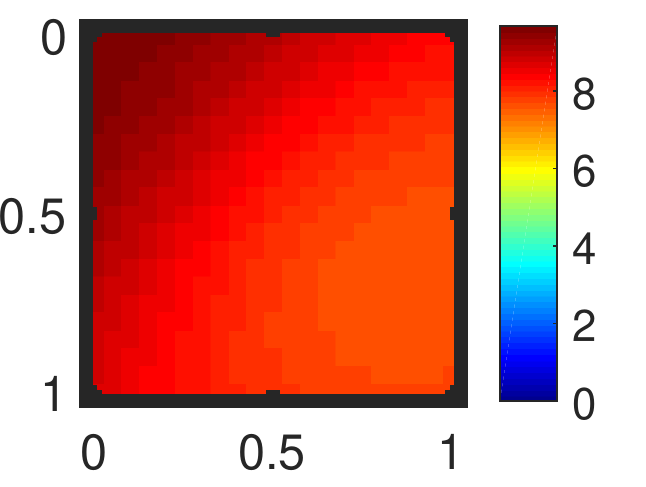}
        \subcaption{$\norm{r_{\theta,\eta}}_1$ vs $\theta$}
        \label{fig::exp-3::25x25::L-1}
    \end{subfigure}
    \begin{subfigure}[b]{.31\linewidth}
        \centering
        \includegraphics[width=\textwidth,height=\textheight,keepaspectratio]{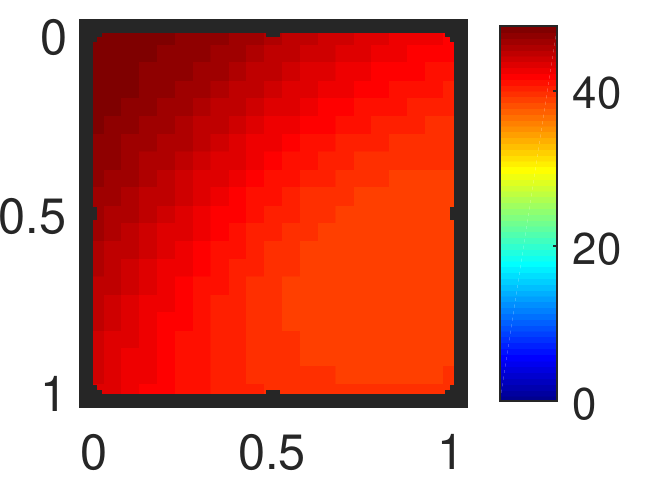}
        \subcaption{$\norm{r_{\theta,\eta}}_2$ vs $\theta$}
        \label{fig::exp-3::25x25::L-2}
    \end{subfigure}
    \caption{Results from Experiment 3. In these plots $L\colon \Rea^{25\times 25} \rightarrow \Rea^{25\times 25}$. See Table \ref{table::exp-3::results} for the contrast, $\theta^s$, $\theta^1$, and $\theta^2$.}
    \label{fig::exp-3::25x25}
\end{figure}

\begin{table}[htb]
    \centering
    \begin{tabular}{c||c|c|c}
         &$\theta^s$&$\theta^1$&$\theta^2$ \\\hline \hline
         $\bb \in \Rea^{100\times100}$&(0.55,0.55)&(0.55,0.55)&(0.55,0.55)\\ \hline
         $\bb \in \Rea^{75\times75}$&(0.55,0.55)&(0.55,0.60)&(0.55,0.60)\\ \hline
         $\bb \in \Rea^{50\times50}$&(0.55,0.50)&(0.55,0.65)&(0.60,0.60)\\ \hline
         $\bb \in \Rea^{25\times25}$&(0.45,0.70)&(0.75,0.90)&(0.70,0.90)
    \end{tabular}
    \begin{tabular}{c||c|c|c}
         Contrast&$\struc{r_{\theta,\eta}}$&$\norm{r_{\theta,\eta}}_1$&$\norm{r_{\theta,\eta}}_2$ \\\hline \hline
         $\bb \in \Rea^{100\times100}$&0.7044&0.3155&0.3215\\ \hline
         $\bb \in \Rea^{75\times75}$&0.5877&0.2876&0.2931\\ \hline
         $\bb \in \Rea^{50\times50}$&0.3677&0.2337&0.2376\\ \hline
         $\bb \in \Rea^{25\times25}$&0.1116&0.1198&0.1125
    \end{tabular}
    \caption{Results from Experiment 3. The above two tables record the location of the minimizer and contrast. Closer to $(0.5,0.5)$ is better for $\theta,$ and the larger the contrast the better.}
    \label{table::exp-3::results}
\end{table}

Throughout all trials of this experiment, $\theta^s$ was closer to $\hat \theta$ then $\theta^1$ or $\theta^2$. Additionally, the contrast is highest when the $\struc{\cdot}$ is used, except when $\mathcal{B}_{4\Delta y_0} \cong \Rea^{25\times25}$. These results show also that the degree to which the problem is overdetermined is indeed important. The more overdetermined the problem, the more nearly $\struc{r_{\theta,\eta}}$ is minimized at $\hat \theta$. Further, the more overdetermined the problem the greater the contrast of $\struc{\cdot}$ relative to $\norm{\cdot}_1$ or $\norm{\cdot}_2$. When $\mathcal{B}_{\Delta y_0} \cong \Rea^{100\times 100}$,  $\contrast(\struc{r_{\theta,\eta}})$ is more than twice either $\contrast(\norm{r_{\theta,\eta}}_1)$ or $\contrast(\norm{r_{\theta,\eta}}_2)$. The ratio of $\contrast(\struc{r_{\theta,\eta}})$ to either $\contrast(\norm{r_{\theta,\eta}}_1)$ or $\contrast(\norm{r_{\theta,\eta}}_2)$ decreases as $L$ becomes square, until finally  $\mathcal{B}_{4\Delta y_0} \cong \Rea^{25\times 25}$ and all three contrasts are similar. These results are consistent with Thms. \ref{thm::noise-calc} - \ref{thm::smooth-func}, which together suggest that as $\Delta y$ decreases, the ability of $\struc{\cdot}$ to distinguish between noise and structure increases.

\section{Conclusion} \label{sec::conclusion}

In this work we have developed a new functional called the structure, which is suitable for detecting forward operator error as it arises in inverse problems. The structure is defined by use of the Earth Mover's Distance (EMD), using a very rapid algorithm and a homogeneous degree one distance. The structure takes as input the residual from an existing inverse procedure, and can be computed quickly. We prove some apparently new results concerning the treatment of noise by EMD. Further, we consistent with these theoretical results we perform numerical experiments and show that the structure is able to distinguish between error in the modeling of a forward operator, and noise in the signal of an inverse problem.

Our numerical results concern a model linear forward operator. On these problems the structure of the residual is indeed minimized when the correct forward operator is used and. The $L_1$ or $L_2$ norms of the residual are also minimized around the correct forward operator, the structure, however, is more localized and has better contrast around the minimum. Further, we observe that the degree to which the inverse problem is overdetermined is pivotal to the success of our procedure. The more over determined the problem, the more useful the structure. This is borne out by the analysis in the case of linear regularization, as well as the numerical results on more sophisticated problem.

In the future, we will extend our work to more sophisticated non-linear operators and promote our error detecting method into an error correcting method.

\appendix

\section{Proofs}

\begin{proof} [Proof of Proposition \ref{thm::residual-tikhonov}]
    Given $\Phi(\bv;\lambda) = \lambda\norm{\bC \bv}^2_2$, the normal equations for Eq. \ref{eqn::linear-regularized-inverse-problem} are
    \begin{equation}
    \label{eqn::normal_equations}
        (\bL_\theta ^T\bL_\theta + \lambda \bC^T\bC)\tilde \bu_{\theta,\eta} = \bL_\theta^T (\bb  + \bseta).
    \end{equation}
    Therefore $\tilde \bL^{-1}_\theta = (\bL_\theta^T \bL_\theta + \lambda \bC^T \bC^T)^{-1} \bL_\theta^T$.  Using the GSVD in Eq. \eqref{eq::LC-GSVD}, a direct calculation gives
    \begin{equation}
        \bL_\theta \tilde \bL_\theta^{-1} = \bU_\theta \bD_{\theta,\lambda} \bU_\theta^T, \quad 
       \text{where $\bD_{\theta,\lambda} := \frac{ {\bf\Sigma}^2_{\theta} } { {\bf\Sigma}_{\theta}^2 + \lambda {\bf\Gamma}_{\theta}^2} \in \Rea^{n \times n}$.}
    \end{equation}
    Thus according to the definition of the residual in Eq. \ref{eqn::inovation-definition}, 
    \begin{align}
        \label{eqn::residual-expansion-correct-op}
        \br_{ \theta,\eta} &= (\bI - \bL \tilde \bL^{-1} ) ( \bb + \bseta)  
            =  \bU_\theta \hat \bD_{\theta,\lambda} \bU_\theta ^T (\bb + \bseta)
            + (\bI - \bU_\theta\bU_\theta ^T)(\bb + \bseta)
    \end{align}
    where
    \begin{equation} 
       \hat \bD_{\theta,\lambda}  
        := (\bI - \bD_{\theta,\lambda})
        = \frac{\lambda {\bf\Gamma}^2_{\theta}} {{\bf\Sigma}_{\theta}^2 + \lambda {\bf\Gamma}_\theta^2} > 0.
    \end{equation}
    We first bound two of the deterministic components of the residual.  
    Using the GSVD,
    \begin{align}
        \bU_\theta \hat \bD_{\theta,\lambda}  \bU_\theta ^T \bb 
        &= \bU_\theta \hat \bD_{\theta,\lambda}  \bU_\theta ^T \bL_\theta \bu
            + \bU_\theta \hat \bD_{\theta,\lambda}  \bU_\theta ^T (\bb - \bL_\theta \bu) \nonumber \\
        &=  \bU_\theta \hat \bD_{\theta,\lambda}  {\bf\Sigma}_\theta \bZ_\theta^T \bu 
            + \bU_\theta \hat \bD_{\theta,\lambda}  \bU_\theta^T (\bb - \bL_\theta \bu).
    \end{align}
     Since  $\norm{\hat \bD_{\theta,\lambda}}_2 \leq 1$  and $\bU_\theta$ is orthogonal, it follows that
    \begin{equation}
        \label{eq::projection_error}
        \norm{\bU_\theta \hat\bD_{\theta,\lambda} \bU_\theta^T (\bb - \bL_\theta \bu)}^2_2 \leq  \norm{(\bb - \bL_\theta \bu)}^2_2
    \end{equation}
    Furthermore, since
    \begin{equation}
        \hat \bD_{\theta,\lambda}   {\bf\Sigma}_\theta
            = \frac{\lambda {\bf\Gamma}^2_{\theta}{\bf\Sigma}_\theta}{{\bf\Sigma}_{\theta}^2 + \lambda {\bf\Gamma}_\theta^2}
            \leq \frac12 \sqrt{\lambda} {\bf\Gamma}_\theta 
            \leq \frac12 \sqrt{\lambda} \bI
    \end{equation}
    (where the inequalities between the diagonal matrices above are interpreted element-wise), it follows that
    \begin{equation}
    \label{eq::regularization_error}
        \norm{\bU_\theta \hat \bD_{\theta,\lambda}  {\bf\Sigma}_\theta \bZ_\theta^T \bu}^2_2
            \leq \norm{\hat \bD_{\theta,\lambda}   {\bf\Sigma}_\theta}^2_2\norm{\bZ_\theta^T \bu}^2_2
            \leq \frac{1}{4}\lambda \norm{\bZ_\theta^T \bu}^2_2.
    \end{equation}
    We next bound the noise component of the residual. Let $\bW_\theta \in \bbR^{m \times (m-n)}$ be a matrix such that $\bQ := (\bU_\theta | \bW_\theta) \in \bbR^{m \times m}$
    is orthogonal and set
    \begin{equation}
       \bsalpha = \begin{pmatrix}
            \bsalpha_\parallel  \\ \bsalpha_{\perp}
                \end{pmatrix} := \bQ^T \bseta  
            = \begin{pmatrix}
            \bU^T_\theta \bseta  \\ \bW_\theta^T \bseta
                \end{pmatrix}.
    \end{equation}
    Then 
    \begin{equation}
         \norm{(\bI - \bL \tilde \bL^{-1} ) \bseta}^2_2
         = \norm{\bU_\theta \hat \bD_{\theta,\lambda} \bU_\theta ^T \bseta + (\bI - \bU_\theta \bU_\theta ^T) \bseta}^2_2
         = \norm{\bU_\theta \hat \bD_{\theta,\lambda} \bsalpha_\parallel}^2_2   + \norm{\bW_\theta \bsalpha_{\perp}}^2_2,
    \end{equation}
    where the last equality uses the fact that the columns of $\bU_\theta$ and $\bW_\theta$ are orthogonal and $\bI - \bU_\theta \bU_\theta ^T = \bW_\theta \bW_\theta^T$.
    Due to the spherical symmetry assumption on $\bseta$, $\bsalpha_\parallel$ and $\bsalpha_{\perp}$ are spherically symmetric random variables of dimension $n$ and $m-n$, respectively, with components that are independent. Therefore
    \begin{align}
        \expec{\norm{\bU_\theta \hat \bD_{\theta,\lambda}\bsalpha_\parallel}^2_2} 
            &= \expec{\norm{\hat \bD_{\theta,\lambda}\bsalpha_\parallel}^2_2} \nonumber\\
            &= \sum^n_{i = 1}\left ( \frac{\lambda \gamma^2_i}{\sigma^2_i + \lambda \gamma^2_i} \right )^2 \expec{\bseta^2_i} 
            = \frac{1}{m}\Tr(\hat{\bD}^2_{\theta,\lambda}) \expec{\norm{\bseta}^2_2}
    \end{align}
    and 
    \begin{equation}
        \expec{ \bW_\theta \norm{\bsalpha_\perp}^2_2} = \expec{\norm{\bsalpha_\perp}^2_2} = \frac{m-n}{m}\expec{\norm{\bseta}^2_2}.
    \end{equation}
    This completes the proof.
\end{proof}

\begin{proof}[Proof of Proposition \ref{thm::EMD-struc-corr-proof}]
It is convenient to write \eqref{eqn::benier-benamou} in the abstract form 
    \begin{equation}
        \EMD(\rho_1,\rho_2) = \min_{m \in C(\rho_1,\rho_2)} {\cT(m)}. 
    \end{equation}
In addition, for any $f \in L^1(\Omega)$, let $m_f$ be a minimizer of $\cT(f^+,f-)$ over $C(f^+,f-)$ so that $\struc{f} = \cT(m_f)$. 
    \begin{enumerate}
    \item We check  absolute homogeneity, positivity, and the triangle inequality.
    \begin{enumerate}
        \item 
            To check absolute homogeneity, let $\lambda \in \bbR$ be a nonzero scalar.  By linearity, $m \in C(|\lambda| f,|\lambda| g)$ if and only if $ |\lambda|^{-1} m \in C(f, g)$.  Therefore
            \begin{multline}
            \label{eqn::EMD_abs_hom}
                \EMD(|\lambda| f, |\lambda| g) 
                    = \min_{m \in C(|\lambda|f,|\lambda| g)} \cT(m) \\
                    = \min_{m \in C(f,g)} \cT(|\lambda| m)  
                    = |\lambda| \min_{m \in C(f,g)} \cT(m) 
                    = |\lambda| \EMD(f,g),
            \end{multline}
            If $\lambda > 0$, \eqref{eqn::EMD_abs_hom} implies that
            \begin{equation}
                \struc{\lambda f} = \EMD(\lambda f^+, \lambda f^-) = |\lambda| \EMD(f^+, f^-) = |\lambda| \struc{f}
            \end{equation}
            If $\lambda < 0$, then $(\lambda f)^{\pm} = |\lambda| f^{\mp}$.  Again \eqref{eqn::EMD_abs_hom} implies that
            \begin{multline}
                \struc{\lambda f} = \EMD((\lambda f)^+, (\lambda f)^-) = \EMD(|\lambda| f^-, |\lambda| f^+)  \\
                = |\lambda| \EMD(f^-, f^+) = |\lambda| \EMD(f^+, f^-) = |\lambda| \struc{f}.
            \end{multline}
            Finally, if $\lambda=0$, then the fact that $\struc{\lambda f} = \lambda \struc{f} = 0$ is trivial.
    
        \item Positivity follows immediately from the positivity of $\EMD$.
        \item  The triangle inequality follows from the fact that 
        \begin{equation}
            (f + g)^+ - (f + g)^- = (f^+ - f^-) + (g^+  - g^-) 
        \end{equation}
        for all $f, g \in L^1(\Omega)$. Thus if $m_f \in C(f^+,f^-)$ and $m_g \in C(g^+,g^-)$, then $m_f + m_g \in C\left ((f+g)^+,(f+g)^-\right)$. Along with the triangle inequality for $\cT$, this implies that
        \begin{align}
            \struc{f + g} &\equiv \cT(m_{f+g}) 
            \leq  \cT(m_{f} + m_g) 
            \leq \cT(m_{f})  +  \cT(m_g)
            \equiv \struc{f} + \struc{g}.
        \end{align}
    \end{enumerate}
    
    \item Because $\frac{1}{\norm{\Omega}}\int_\Omega (g + c) dx = \frac{1}{\norm{\Omega}}\int_\Omega g dx + c,$ we have that $g^+ = (g + c)^+,$ and $g^- = (g + c)^-.$ Therefore 
    \begin{equation}
    \label{eq::struc-const}
        \struc{g + c} = \EMD\left ((g + c)^+,(g + c)^-\right) = \EMD(g^+,g^-) = \struc{g}.
    \end{equation}
    
    \item Let $g = 0$ in \eqref{eq::struc-const} above.  Then
    \begin{align}
    \label{eq::struc-const-zero}
        \struc{c} = \struc{0} = 0, \quad \forall c \in \Rea.
    \end{align}
    
    \item Because the constraint in Eq. \ref{eqn::benier-benamou} involves only the difference of $\rho_1$ and $\rho_2$, it follows that $\EMD(\rho_1,\rho_2) = \EMD(\rho_1 + f,\rho_2 + f)$ for any non-negative $f \in L^1(\Omega)$.   Moreover, because $\rho_2$ and $\rho_1$ have the same mass, the average of $\rho_2 - \rho_1$ is zero.  Hence,
    \begin{align}
    \label{eq::struc-is-EMD-intermediate}
        \struc{\rho_2 - \rho_1} &= \EMD(\max(\rho_2 - \rho_1,0),\max(\rho_1 - \rho_2,0)) \nonumber \\
        &= \EMD(\max(\rho_2 - \rho_1,0) + \min(\rho_1,\rho_2),\max(\rho_1 - \rho_2,0) + \min(\rho_1,\rho_2))
    \end{align}
    Since $\forall x,y \in \Rea, \max(x - y,0) + \min(x,y) = x$, it follows from \eqref{eq::struc-is-EMD-intermediate} that
    \begin{equation}
    \label{eq:struc-is-EMD}
        \struc{\rho_2 - \rho_1} = \EMD(\rho_2,\rho_1) = \EMD(\rho_1,\rho_2) 
    \end{equation}
    
    \end{enumerate}
\end{proof}

\

Before proving Thm. 1-3, we will first prove two useful lemmas, which will be used extensively.

\begin{lemma}[$\EMD$ triangle inequality]\label{lem::triangle-inequality-emd}
Let $\Omega\subset \Rea^n$ be a bounded set and $f$, $g$, $h \in L^{\infty}(\Omega)$ and $\int_{\Omega}f dx = \int_{\Omega}h dx  = \int_{\Omega}g dx$. Then 
\begin{equation}
    \EMD(f,g) \leq \EMD(f,h) + \EMD(h,g).
\end{equation}
\end{lemma}

\begin{proof}
    Recall from Prop. \ref{thm::EMD-struc-corr-proof} that $\struc{f - g} = \EMD(f,g)$, then by the triangle inequality of $\struc{\cdot}$,
    \begin{align}
        \EMD(f,g) = \struc{f - g} \leq \struc{f - h} + \struc{h - g} = \EMD(f, h) + \EMD(h, g)
    \end{align}
\end{proof}

\begin{lemma}[$\struc{\cdot}$ and $\EMD$ of the mean] \label{lem::struc-and-emd-of-mean}
    $\Omega\subset \Rea^n$ be a bounded set and $f \in L^{\infty}(\Omega)$ and $\mu = \frac{1}{|\Omega|}\int_{\Omega}f dx$. Then
    \begin{equation}
        \struc{f} = \EMD(f , \mu).
    \end{equation}
\end{lemma}

\begin{proof}
    Recall from Prop. \ref{thm::EMD-struc-corr-proof} that $\EMD(f,g) = \EMD(f + h, g + h),$ therefore
    \begin{align}
        \struc{f} = \EMD(f^+, f^-) = \EMD(f^+ +(\mu - f^-), f^- +(\mu - f^-)) = \EMD(f , \mu).
    \end{align}
\end{proof}

\begin{lemma}[$\EMD$ Subadditivity]
\label{lem::EMD-subadditivity}
    If $\EMD(f_1,g_1)$ and $\EMD(f_2, g_2)$ are well defined, then so too is $\EMD(f_1 + f_2, g_1 + g_2)$, and 
    \begin{equation}
        \EMD(f_1 + f_2, g_1 + g_2) \leq \EMD(f_1, g_1) + \EMD(f_2, g_2).
    \end{equation}
\end{lemma}

\begin{proof}
    We use the Eq. \ref{eqn::emd-definition} of the $\EMD$. Let $\pi_1$ and $\pi_2$ satisfy the constraint of Eq. \ref{eqn::wasserstein-distance-definition} for $\EMD(f_1, g_1)$ and $\EMD(f_2, g_2)$ resp. Then clearly
    \begin{align}
        \int_{\Omega} (\pi_1 + \pi_2)dx^{(2)} &= f_1 + f_2 \nonumber\\
        \int_{\Omega} (\pi_1 + \pi_2)dx^{(1)} &= g_1 + g_2 \nonumber\\
        \pi_1  + \pi_2 \geq 0,
    \end{align}
    and so by the minimality of the $\EMD$,
    \begin{align}
        \EMD(f_1, g_1) + \EMD(f_2, g_2) &= \int_{\Omega\times\Omega}c \pi_1 dx^{(1)} dx^{(2)} + \int_{\Omega\times\Omega}c \pi_2 dx^{(1)} dx^{(2)} \nonumber\\
        &= \int_{\Omega\times\Omega}c (\pi_1 + \pi_2) dx^{(1)} dx^{(2)} \nonumber\\
        &\geq \min_{\pi \geq 0} \int_{\Omega\times\Omega}c \pi dx^{(1)} dx^{(2)}\nonumber\\
        &= \EMD(f_1 + f_2, g_1 + g_2)
    \end{align}
    where $\pi$ is subject to the constraints of Eq. \ref{eqn::wasserstein-distance-definition} where $\rho_1 = f_1 + f_2$ and $\rho_2 = g_1 + g_2$.
\end{proof}

\begin{lemma}[$\EMD$ is bounded by the $L_1$ norm] \label{lem::EMD-bound-by-L1-norm}
    Let $\Omega$ be a bounded set, and $l \geq \norm{x^{(1)} - x^{(2)}}_2$ for all $x^{(1)}, x^{(2)} \in \Omega$. If $f,g : \Omega \rightarrow \Rea^+$ then 
    \begin{equation}
        \EMD(f,g) \leq \frac{l}{2} \norm{f - g}_{L^1(\Omega)}.
    \end{equation}
\end{lemma}

\begin{proof}
    Let $\gamma = \int_{\Omega}{(f - g)^+}dx$ and $\xc$ be such that $\norm{\xc - x}_2 \leq l/2$ $\forall x \in \Omega$ then
    \begin{align}
        \EMD(f,g) &= \struc{f - g} \leq \EMD((f-g)^+,\gamma \delta_{\xc}) + \EMD(\gamma \delta_{\xc}, (f-g)^-)\nonumber\\
        & \leq \frac{l}{2} \norm{(f - g)^+}_{L^1(\Omega)} + \frac{l}{2} \norm{(f - g)^-}_{L^1(\Omega)} 
        = \frac{l}{2} \norm{f - g}_{L^1(\Omega)}
    \end{align}
The last two lines could use a few details between them.
\end{proof}

\begin{lemma}[Expectation bound by the standard deviation]
	 \label{lem::expec-abs-is-less-than-stdev}
    Let $\eta$ be a scalar random variable with zero mean such that $\var[\eta]$ is finite. Then
   $\expec{|\eta|} \leq \sqrt{ \var[\eta]}$.
\end{lemma}

\begin{proof}
     Let $\psi$ be the probability distribution for $\eta$. By the Cauchy-Schwarz inequality,
    \begin{align}
       \expec{|\eta|} \equiv \int^\infty_{-\infty} |x|\psi(x)dx 
       \leq \left (\int^{\infty}_{-\infty} x^2\psi(x) dx \right)^\frac{1}{2} \left (\int^\infty_{-\infty} \psi(x) dx \right )^\frac{1}{2} 
       = \big(\var[\eta] \big)^{1/2}.
    \end{align}

\end{proof}

\

We now proceed to the proof of Theorem 2, but first it is helpful to give a brief summary. To bound the EMD from above, we give a candidate transport plan that is based on the multigrid strategy depicted in Fig. \ref{fig::proof-aid} for the case $d=2$.
In this case, the strategy is to divide the domain into square windows with two square panels per side, as shown in Figure \ref{fig::proof-aid}. The mass in each window is then redistributed in such a way that the new distribution is constant on each window. Each window then becomes a panel in a window that is a factor a factor of two larger in each dimension, and the process is repeated until the distribution on the entire square is constant.   For $d >2$, the plan is the same, except that each window is a hypercube $2^d$ panels.  The cost of the complete transport plan can be bounded by the sum of the costs of the transport plan for each step.  These costs are computed in the proof below and their sum leads to the bound in Theorem \ref{thm::noise-calc}.

\begin{figure}
    \centering
    \begin{subfigure}[b]{.3\linewidth}
        \centering
        \includegraphics[width=\textwidth,keepaspectratio]{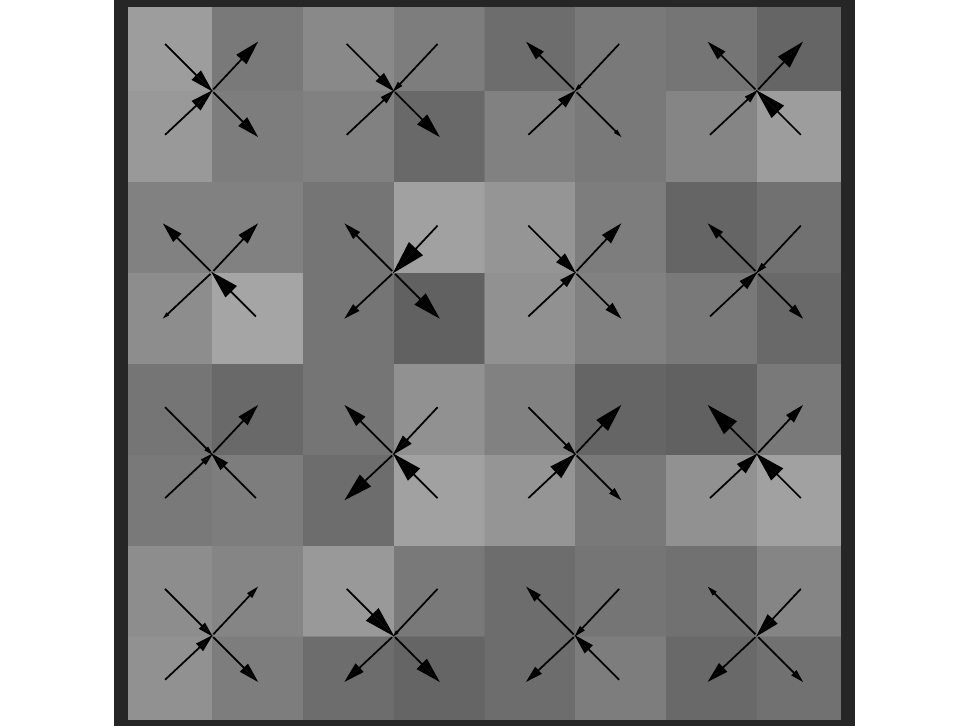}
        \subcaption{$H_3$}
        \label{fig::proof-aid::8-panels-per-side}
    \end{subfigure}
    \begin{subfigure}[b]{.3\linewidth}
        \centering
        \includegraphics[width=\textwidth,keepaspectratio]{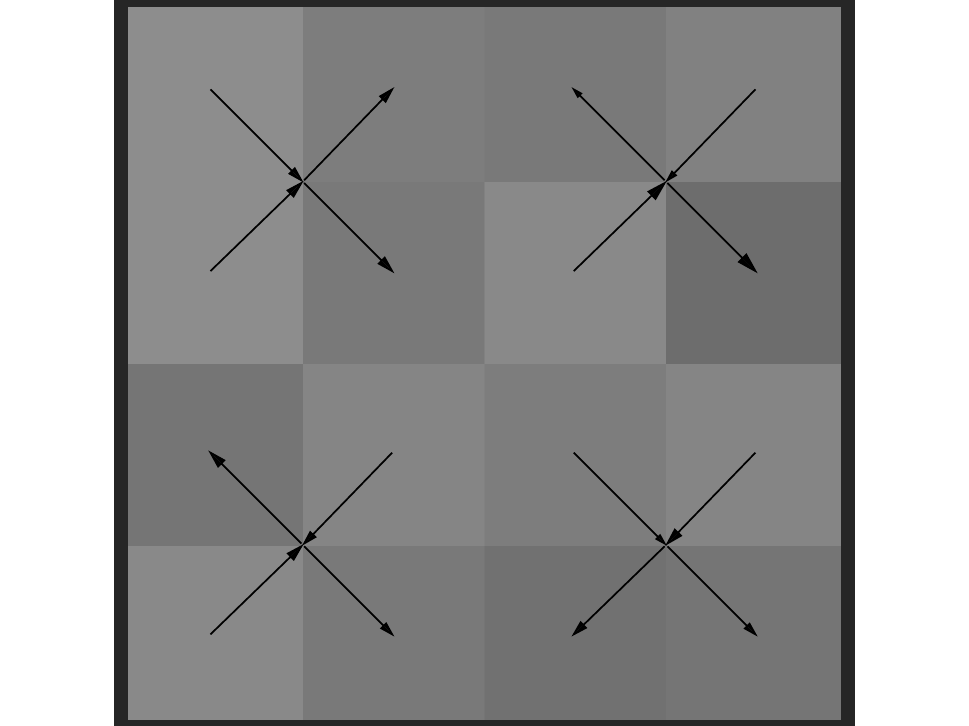}
        \subcaption{$H_2$}
        \label{fig::proof-aid::4-panels-per-side}
    \end{subfigure}
    \begin{subfigure}[b]{.3\linewidth}
        \centering
        \includegraphics[width=\textwidth,keepaspectratio]{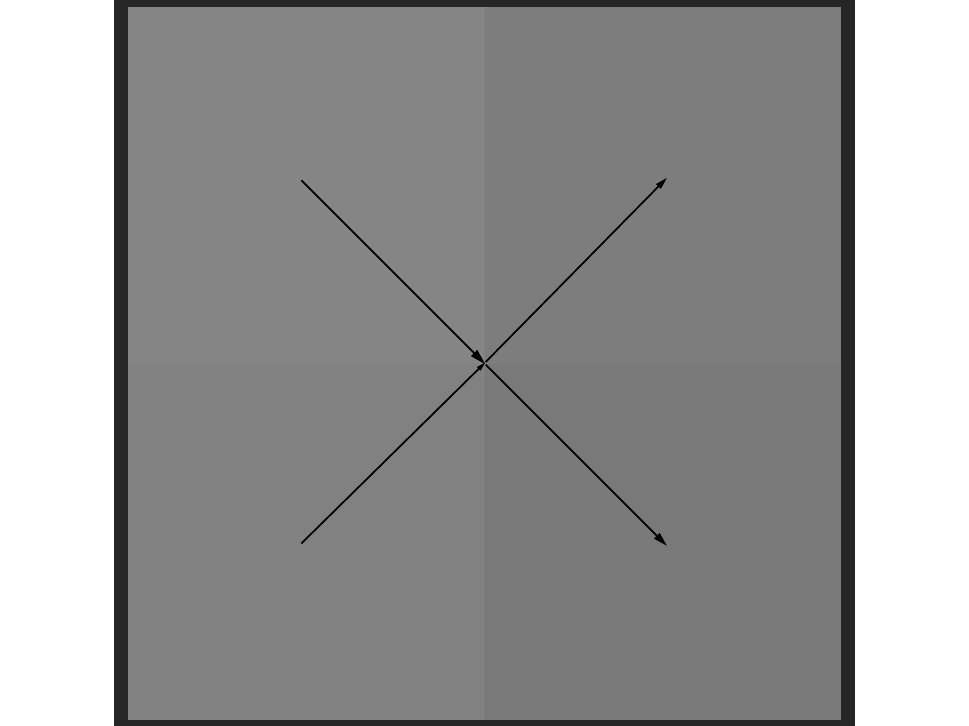}
        \subcaption{$H_1$}
        \label{fig::proof-aid::2-panel-per-side}
    \end{subfigure}
    \caption{The multigrid idea of Theorem 1 when $\ell=3$.  At each step, a transport plan is computed in each $2$x$2$ window.  Then the same problem is solved at the next coarser scale. In the above figures, the arrow tip area is proportional to the mass transported at each substep. The function $H_i$ is defined in \eqref{eqn::H-i-definition}.}
    \label{fig::proof-aid}
\end{figure}

\begin{proof}[Proof of Theorem 2] \label{sec::proof-of-claims::thm-1-proof}

Since $\struc{h_\ell} = \struc{h_\ell - \bar{\mu}}$ we can assume, without loss of generality, that $\bar \mu = 0$. Consider  the case $\ell = 1$, which will be used for the general setting later.  We construct a two-step plan that first moves all of the mass in $h_1^+$ to the point $\yc = ( 1/2, \dots, 1/2)$ at the center of the domain and then moves the mass from $\yc$ to $h_1^-$.\footnote{While the definition of the EMD in Eq. \ref{eqn::emd-definition} is still well-defined for delta function, the formula in Eq. \ref{eqn::benier-benamou} is not.  Thus while we use Eq. \ref{eqn::benier-benamou} for numerical calculations, we often rely on Eq. \ref{eqn::emd-definition} for theoretical bounds.}

Let $\gamma = \int_{\Omega} h^+_1 dy = \int_{\Omega} h^-_1 dy$, $\mu_0 = \int_{\Omega} h_1 dy$, and $\gamma_{1,k} = |\eta_{1,k} - \mu_0| |\omega_{1,k}|$. 
Then $\EMD(h^+_1,\gamma\delta_{\yc}) = \EMD(\gamma\delta_{\yc},h^-_1)$ and
\begin{align}
    \label{eqn::emd-calc-sum-of-panels}
    \struc{h_1} \equiv \EMD(h_1^+, h_1^-) \nonumber
    &\leq \EMD(h_1^+, \gamma \delta_{\yc}) + \EMD(\gamma \delta_{\yc}, h_1^-) \nonumber \\
    & = \sum_{k = 1}^{2^{d}} \EMD \left(|\eta_{1,k} - \mu_0|\chi_{1,k}, \gamma_{1,k}\delta_{\yc}\right).
\end{align}
Thus we turn our attention to computing the terms in the sum above. First,
\begin{align}
\label{eqn::panel_constant_scaling}
    \EMD(|\eta_{1,k} - \mu_0|\chi_{1,k}, \gamma_{1,k} \delta_{\yc}) = |\eta_{1,k} - \mu_0|\EMD(\chi_{1,k},|\omega_{1,k}| \, \delta_{\yc}).
\end{align}
There is only one one admissible transport plan (see from Eq. \ref{eqn::emd-definition}) between $\chi_{1,k}$ and $|\omega_{1,k}|\delta_{\yc}$; it simply moves the mass around each point of $\omega_{1,k}$ to $\yc$: 
\begin{equation}
    \pi\left(x^{(1)},x^{(2)}\right) =  \chi_{1,k} (x^{(1)}) \times \delta_{\yc}(x^{(2)}) 
\end{equation} 
 If we consider the more general case where $\omega_{1,k}$ has side length $l$, then upon a change of coordinates,

\begin{align}
    \label{eqn::panel_tp}
    \EMD(\chi_{1,k},|\omega_{1,k}| \delta_{\yc}) &= \int_\Omega \int_\Omega \norm{x^{(1)} - x^{(2)}}_2 \chi_{1,k} (x^{(1)}) \times \delta_{\yc}(x^{(2)}) dx^{(1)} dx^{(2)}\nonumber\\ 
    &= \int_{\omega_{1,k}}\int_\Omega \norm{x^{(1)} - x^{(2)}}_2\delta_{\yc}(x^{(2)}) dx^{(1)} dx^{(2)} \nonumber \\
    &= \int_{\omega_{1,k}} \norm{x^{(1)} - \yc}_2 dx^{(1)} = \int_{\left [0,l\right ]^d}\norm{x^{(1)}}_2 dx^{(1)} \nonumber\\
    &\leq \sqrt{d} \int_{\left [0,l\right ]^d} \norm{x^{(1)}}_\infty dx^{(1)} 
    \leq \sqrt{d} \frac{l^{d+1}}{2}\nonumber\\
\end{align}
%




Substituting \eqref{eqn::panel_constant_scaling} and \eqref{eqn::panel_tp} into \eqref{eqn::emd-calc-sum-of-panels} gives

\begin{align}
    \label{eqn::emd-calc-sum-of-vars}
    \struc{h_1} 
    &\leq \sum_{i=1}^{2^d} |\eta_{1,k} - \mu_0|\frac{\sqrt{d}l^{d+1}}{2} = \frac{\sqrt{d}}{2^{d+2}} \sum_{k=1}^{2^d} |\eta_{1,k} - \mu_0|,
\end{align}
where we have used the fact that when $\ell = 1, l = 2^{-1}$.  
A standard calculation shows that
\begin{equation}
    \label{eqn::variance-of-abs-of-var}
    \var(|\eta_{1,k} - \mu|) \leq \var(|\eta_{1,k}|),\quad i = 1,\dots,2^{d}.
\end{equation}
Further, w.l.o.g. $\expec{\eta_{1,k}} = 0$ and Lemma \ref{lem::expec-abs-is-less-than-stdev} give:
\begin{equation}
    \label{eqn::expec-of-abs}
    \expec{|\eta_{1,k} - \mu_0|} \leq \sigma
\end{equation}
with Eq. \ref{eqn::emd-calc-sum-of-vars} and get

\begin{align}
    \expec{\struc{h_1}} &\leq \frac{\sqrt{d} 2^{d}}{2^{(d+2)}} \sum_{k = 1}^{2^d} \expec{|n_{1,k} - \mu_0|} 
    \leq  \frac{\sqrt{d} 2^{d}}{2^{(d+2)}} \sigma = \frac{\sqrt{d}}{4}\sigma.
\end{align}
Now we consider the case when $\ell > 1$. Define the functions
\begin{align}
    \label{eqn::H-l-definition}
    H_\ell(y) &= h_\ell(y) = \sum^{2^{\ell d}}_{k = 1} \eta_{\ell,k}\chi_{\ell,k}(y)\\
    \label{eqn::H-i-definition}
    H_i(y) &= \sum^{2^{id}}_{k = 1} \mu_{i,k} \chi_{i,k}(y), \text{ where } \mu_{i,k} = \frac{1}{|\omega_{i,k}|}\int_{\omega_{i,k}} H_{i+1}(y) dy, \quad i = 0,1,\dots,\ell-1.
\end{align}
Instances of $H_i$ are shown in Fig. \ref{fig::proof-aid}.  
The function $h_\ell$ can be written as the telescoping sum
\begin{align}
    \label{eqn::telescope-sum}
    h_\ell = H_\ell
    = (H_\ell - H_{\ell - 1}) + (H_{\ell - 1} - H_{\ell - 2}) + \dots + (H_2 - H_1)  + (H_1 - H_0) + H_0.
\end{align}
Moreover, because  $H_i = \sum^{2^{d(i-1)}}_{k = 1}H_i\chi_{i - 1,k}$, it follows that
\begin{equation}
\label{eqn::sik}
H_i - H_{i-1} =  \sum^{2^{d(i-1)}}_{k = 1}s_{i-1,k},
\quad \text{where }
     s_{i-1,k}(y) = \left (H_i(y) - \mu_{i-1,k} \right )\chi_{i-1,k}(y).
\end{equation}
We apply $\struc{\cdot}$ to \eqref{eqn::telescope-sum}, using \eqref{eqn::sik}, the triangle inequality, and the fact that $\struc{H_0}=0$ (because it is a constant). The result is
\begin{align}
    \label{eqn::struc-h-ell}
\struc{h_\ell} 
\leq \sum^{\ell}_{i = 1} \struc{H_i - H_{i - 1}} 
\leq \sum^{\ell}_{i = 1} \sum^{2^{d(i - 1)}}_{k = 1} \struc{s_{i-1,k}}.
\end{align}
To evaluate $\struc{s_{i-1,k}}$, we repeat the argument used to generate Eq. \ref{eqn::emd-calc-sum-of-vars}. This gives

\begin{align}
    \label{eqn::struc-sik}
    \struc{s_{i-1,k}} 
    \equiv \EMD(s_{i-1,k}^+, s_{i-1,k}^-) 
    \leq \frac{\sqrt{d}l^{d+1}}{2} \sum_{k':\omega_{i,k'} \subset \omega_{i-1,k}} |\mu_{i,k'} - \mu_{i-1,k}|.
\end{align}
By construction, 
\begin{equation}
\label{eq::mu-avg}
    \mu_{i-1,k} = 2^{-d} \sum_{k':\omega_{i,k'} \subset \omega_{i-1,k}} \mu_{i,k}.
\end{equation}
It follows that 
the random variable $(\mu_{i+1,k'} - \mu_{i,k})$ that appears in \eqref{eqn::struc-sik} has zero mean.  Thus Lemma \ref{lem::expec-abs-is-less-than-stdev} applies and
\begin{align}
    \expec{|\mu_{i,k'} - \mu_{i-1,k}|}
    \leq \left(\var[|\mu_{i,k'} - \mu_{i-1,k}|]\right)^{\frac{1}{2}}
    &\leq \left(\var[|\mu_{i,k'}|]\right)^{\frac{1}{2}} := \sigma_i,
\end{align}
where the last two inequalities above follows from standard probability theory.  Also, because of \eqref{eq::mu-avg}, another standard probablity result gives
\begin{equation}
\label{eq::sigma-i}
     \sigma_{i} = 2^{-\frac{d}{2}} \sigma_{i+1} = \dots = 2^{-\frac{d}{2}(\ell - i)} \sigma_\ell, \quad i = 1,\dots, \ell.
\end{equation}
We now take the expectation of \eqref{eqn::struc-sik}, using the fact that $\omega_{i,k'}$ has side length $l = 2^{-i}$, along with the triangle and \eqref{eq::sigma-i},.  The result is

\begin{align}
    \label{eqn::struc-s}
    \expec{\struc{s_{i-1,k}}} 
    \leq \sqrt{d} 2^{-i(d+1)-1} \sum_{k':\omega_{i,k'} \subset \omega_{i-1,k}} 
             2^{-\frac{d}{2}(\ell - i)} \sigma_\ell
    =  \sqrt{d} 2^{-\frac{id}{2} -i+d- \frac{d\ell}{2} -1} \sigma_\ell
\end{align}
Substituting this bound into \eqref{eqn::struc-h-ell} gives
\begin{align}
\label{eqn::inductive-step-expansion}
   \expec{\struc{h_\ell}}  
   \leq \sum^{\ell}_{i = 1} 
        \sum^{2^{d(i - 1)}}_{k = 1} \sqrt{d} 2^{-\frac{id}{2} -i+d- \frac{d\ell}{2} -1} \sigma_\ell 
 = \frac{\sqrt{d}\sigma_\ell}{2^{1+\frac{\ell d}{2}}}\sum^{\ell}_{i = 1} \left (2^{\frac{d}{2} - 1}\right)^i
\end{align}




%
If $d = 2,$ then $2^{\frac{d}{2} - 1} = 1$ and Eq. \ref{eqn::inductive-step-expansion} becomes 
\begin{align}
    \label{eqn::thm-dim-2-final-answer}
    \expec{\struc{h_\ell}} = \expec{\struc{H_\ell}} &\leq \frac{2\sigma_\ell}{2^{1+i}}\ell = \frac{\sigma_\ell \ell}{2^{\ell}}.
\end{align}
If $d \geq 3$, then $2^{\frac{d}{2}-1}/(2^{\frac{d}{2}-1} -1) \leq 4$, so the geometric sum in Eq. \ref{eqn::inductive-step-expansion} is
\begin{align}
    \label{eqn::geometric-sum-answer}
    \sum^{\ell}_{i = 1} \left (2^{\frac{d}{2} - 1}\right)^i 
     = \frac{2^{\left (\frac{d}{2} - 1\right)(\ell + 1)} - 2^{\frac{d}{2} - 1}}{2^{\frac{d}{2} - 1} - 1}
    \leq \frac{2^{\frac{d}{2} - 1} 2^{\left(\frac{d}{2} - 1\right)\ell} } {{2^{\frac{d}{2} - 1} - 1}}
    \leq 2^{\frac{\ell d}{2} - \ell + 2}.
\end{align}
Thus for $d \geq 3$,
\begin{align}
\label{eqn::struc-noise-estimate-grometric-series}
    \expec{\struc{h_\ell}} 
    &\leq \sqrt{d}\sigma_\ell
    \frac{2^{\frac{\ell \sqrt{d}}{2}  - \ell + 2}}
        {2^{1+\frac{\ell \sqrt{d}}{2}}}
    = \sqrt{d}\sigma_\ell 2^{-\ell+1}
\end{align}


Finally, setting $\epsilon = 2^{-\ell}$ gives
\begin{equation}
    \expec{\struc{h_\ell}} \leq \sigma
    \begin{cases}
        -\epsilon_\ell \log(\epsilon_\ell)&\text{ when $d = 2$}\\
        2\sqrt{d}\epsilon_\ell&\text{ when $d > 2$}\\
    \end{cases}
\end{equation}
This completes the proof.
\end{proof}

\begin{proof}[Proof of Lemma \ref{thm::noise-calc-l2}] The proof follows directly from the definition of $h_\ell$ in the statement of Thm. \ref{thm::noise-calc}:
\begin{align}
   \expec{\norm{h_\ell}_2^2} = \expec{\int_{[0,1)^d} \left (h_\ell(y)\right )^2 dy}= \sum^{2^{\ell d}}_{k = 1} \expec{\eta_{\ell,k}^2} 2^{-\ell d} = 2^{-\ell d}\sum^{2^{\ell d}}_{k=1} \sigma^2= \sigma^2.
\end{align}
\end{proof}

\begin{proof}[Proof of Theorem \ref{thm::smooth-func}]

Without loss of generality, assume that $\phi$ is positive a.e.  (If not, simply replace $\phi$ by $\phi - \essinf{\phi}$ and use \eqref{eq::struc-const}.) By construction, $\phi$ and $R_\ell \phi$ have the same average over $Y$, which we denote by $\mu$. Thus by Lemmas \ref{lem::triangle-inequality-emd} and \ref{lem::struc-and-emd-of-mean},
\begin{align}
    \label{eqn::struc-triangle-sum}
    \struc{R_\ell\phi} = \EMD(R_\ell\phi,\mu) \leq \EMD(R_\ell\phi, \phi) + \EMD(\phi,\mu) = \EMD(R_\ell\phi,\phi) + \struc{\phi}.
\end{align}
Hence
\begin{align}
 \label{eqn::struc-diff-oneside}
    \struc{R_\ell\phi} - \struc{\phi} \leq \EMD(R_\ell \phi, \phi).
\end{align}
One the other hand, switching the roles of $R_\ell\phi$ and $\phi$  Eq. \ref{eqn::struc-triangle-sum} gives
\begin{align}
 \label{eqn::struc-diff-otherside}
    \struc{\phi} - \struc{R_\ell\phi} \leq \EMD(R_\ell \phi, \phi)
\end{align}
Together \eqref{eqn::struc-diff-oneside} and \eqref{eqn::struc-diff-oneside} imply the bound
\begin{align}
    \label{eqn::abs-struc-ri-phi-phi-bound}
    |\struc{R_\ell\phi} - \struc{R_\ell\phi}| \leq \EMD(R_\ell \phi, \phi).
\end{align}
We now bound $\EMD(R_\ell\phi, \phi)$. For any $\ell,i$ $\int_{\omega_{\ell,i}} R_\ell \phi dy = \int_{\omega_{\ell,i}} \phi dy$.  Thus by Lemma \ref{lem::EMD-subadditivity}, 
\begin{equation}
    \label{eqn::EMD-restriction-sum}
    \EMD(R_\ell \phi, \phi) \leq \sum^{2^{\ell d}}_{i = 1} \EMD(R_\ell\phi \chi_{\ell,i}, \phi\chi_{\ell,i})
\end{equation}
and further by Lemma \ref{lem::EMD-bound-by-L1-norm}, for $i = 1, \dots, 2^{\ell d}$ 
\begin{align}
    \label{eq::EMD-L1-loc-bound}
    \EMD(R_\ell\phi \chi_{\ell,i}, \phi\chi_{\ell,i}) \leq \norm{R_\ell\phi - \phi}_{L^1(\omega_{\ell,i})} {d}^{1/2} 2^{-\ell}
\end{align}
Now we bound $\norm{R_\ell\phi - \phi}_{L^1(\omega_{\ell, i})}$.  Since $\phi \in C^1\left (\overline Y\right)$, it follows that, for $y \in \omega_{\ell,i}$
 \begin{align}
   | R_\ell\phi(y) - \phi(y) | 
   &= \frac{1}{|\omega_{\ell,i}|} \left |\int_{\omega_{\ell,i}} (\phi(y') - \phi(y)) dy' \right | \nonumber \\
  &\leq \sup_{y \in \omega_{\ell,i}} |\nabla \phi(y)| \sup_{y \in \omega_{\ell,i}} |y'-y|
  \leq d^{1/2} 2^{-\ell} \sup_{y \in \omega_{\ell,i}} |\nabla \phi(y)| 
 \end{align}
Therefore
\begin{equation}
    \label{eqn::L1-norm-function-difference-taylors}
    \norm{R_\ell\phi - \phi}_{L^1(\omega_{\ell,i})} 
    \leq |\omega_{\ell,i}| d^{1/2}2^{-\ell}\sup_{y \in \omega_{\ell,i}}|\nabla \phi(y)|
    = d^{1/2}2^{-(d+1) \ell}\sup_{y \in \omega_{\ell,i}}|\nabla \phi(y)|.
\end{equation}
Combining Eq. \ref{eqn::abs-struc-ri-phi-phi-bound}, \eqref{eq::EMD-L1-loc-bound}, and Eq. \ref{eqn::L1-norm-function-difference-taylors} yields
\begin{align}
    |\struc{R_\ell\phi} - \struc{\phi}| 
    &\leq \sum^{2^{\ell d}}_{i = 1} d 2^{-(d+2)\ell } \sup_{y \in \omega_{\ell,i}}|\nabla \phi(y)|  
    \leq d2^{-2\ell}\sup_{y \in Y}|\nabla \phi(y)|
    \equiv C(|\nabla \phi|)d\epsilon_\ell^{2},
\end{align}
where $C(|\nabla \phi|) = \sup_{y \in Y} |\nabla \phi(y)|$ and $\epsilon_\ell = 2^{-\ell}$.  This completes the proof.
\end{proof}

\section{Line Integral Operators} \label{sec::LIO}

Recall from Section \ref{sec::background} the spaces $\mathcal{U}$ and $\mathcal{B}$ of functions defined on domains $X$ and $Y$, respectively. An operator $\mathcal{L}\colon \mathcal{U}\rightarrow \mathcal{B}$ is a line integral operators (LIO), if $\forall f \in \mathcal{U},$

\begin{equation}
    \label{eqn::LIO-definition}
    (\mathcal{L}f)(y) = \int_{P_y} f(x) dx = \int^1_0 f(\hat x(t;y))\hat x'(t;y) dt,
\end{equation}
where for each $y \in Y$, $P_y = \{ \hat x(t;y) : t \in (0,1) \} \subset X$, and $\hat x(t;y)$ is continuous in $t$ and $y$. In particular, if $f$ is a continuous on $X$, then $\mathcal{L}f$ is continuous on $Y$. Figs. \ref{fig::LIO-cont::paths} and \ref{fig::LIO-cont::points} illustrate a LIO in two dimensions.  The recipe we used to generate examples of $\hat{x}$ is given below.

\begin{figure}[!ht]
    \centering
        \begin{subfigure}[b]{.32\linewidth}
        \centering
        \includegraphics[width=\textwidth,keepaspectratio]{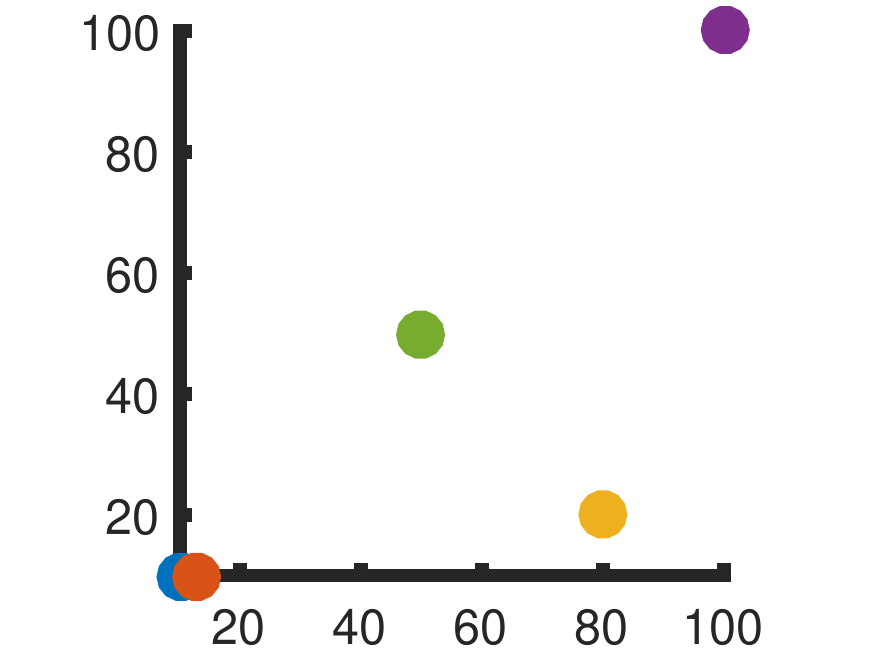}
        \subcaption{The values of $y$.}
        \label{fig::LIO-cont::points}
    \end{subfigure}
    \begin{subfigure}[b]{.32\linewidth}
        \centering
        \includegraphics[width=\textwidth,keepaspectratio]{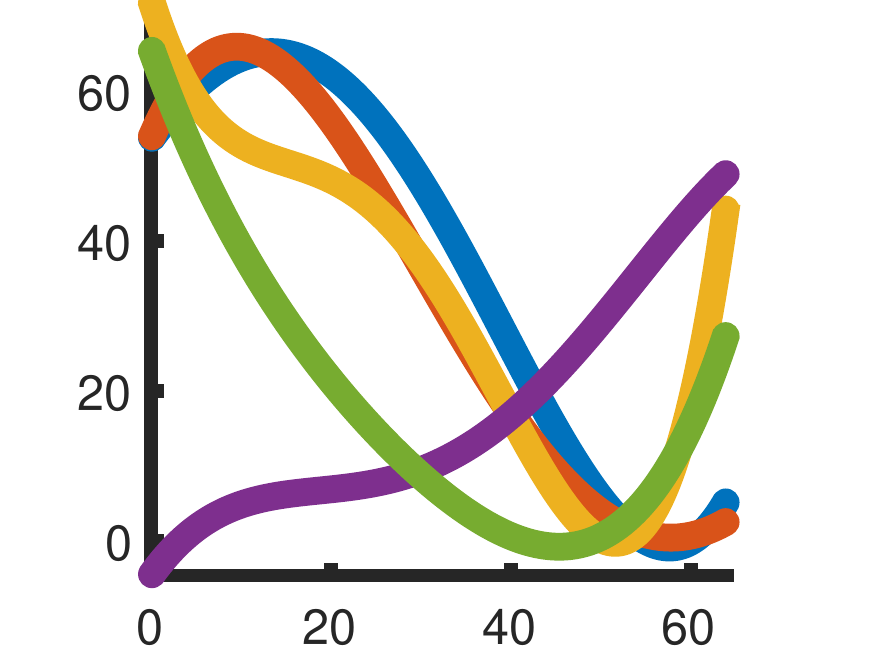}
        \subcaption{The curves $P_y$.}
        \label{fig::LIO-cont::paths}
    \end{subfigure}%
    \begin{subfigure}[b]{.32\linewidth}
        \includegraphics[width=\textwidth,keepaspectratio]{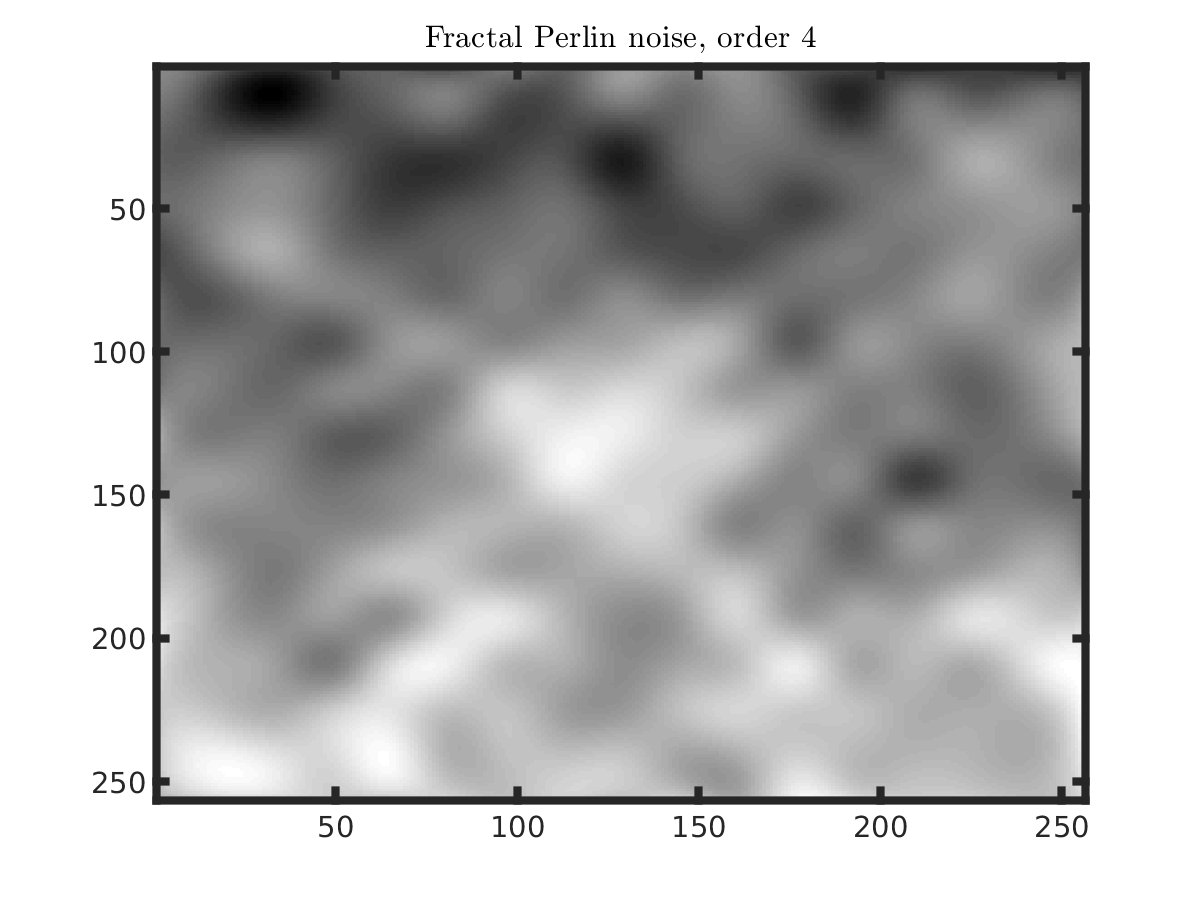}
        \subcaption{Example of Perlin noise.}
        \label{figure::Perlin-noise-order-4}
    \end{subfigure}
    \caption{An example of a LIO. Points on the right are used to generate curves on the left of the same color. Coefficients for the parameterization in \eqref {eqn::x-bar-definition} of $P_y$ come from Perlin noise.}
 
    \label{fig::LIO-cont}
\end{figure}

To discretize $\mathcal{L}$, we generate a path $P_y$ for each hypercube $\omega \subset Y$. Line integrals along these paths are approximated via quadrature. For all LIOs, we use same the quadratures, and $X$, and $Y$.

To construct the LIO for Experiments 1 - 3, we do the following. 

\begin{enumerate}
    \item {\bf Construction of numerical grids.} In all of our computational examples, the domains $X$ and $Y$ are unit squares in $\bbR^2$.  We discretize these domains with $N^x$ and $N^y$ points, respectively, on each side and define grid points
    \begin{subequations}
    \begin{gather}
    \label{eqn::quadrature-approx-definition}
    x_{i,j} = \left ( i\Delta x, j \Delta x\right ),  \quad 0 \leq i,j \leq N^x -1,
   \\
    y_{k,l} = \left ( k\Delta y, l \Delta y\right ), \quad 0 \leq k,l \leq N^y -1,
    \end{gather}
    \end{subequations}
    where $\Delta x = 1/N^x$ and  $\Delta y = 1/ N^y$. 
     We then generate values $u_{i,j}$ by sampling a prescribed function at the points $x_{i,j}$.  An illustrative example is given in Fig. \ref{fig::exp-1::ground-truth}, where piecewise smooth rings have been sampled on a $64\times 64$ grid.
    
    \item {\bf Generation of smooth paths.} To form $\hat x$, we first sample coefficients $\alpha_{p,r}$ for $p = 0,\dots, 4$ and $r = 1,2$ from Perlin noise \cite{perlin1985image,perlin2002improving} of order four. In Fig. \ref{figure::Perlin-noise-order-4}, a realization of one such coefficient as a function of $y$ is shown on a $256\times 256$ grid. Given these coefficients, we let 
    $\bar x =( x^{(1)},x^{(2)})$ be polynomials in $t$:
    \begin{align}
        \label{eqn::x-bar-definition}
        \bar x^{(r)}(t;y_{k,l}) &= \sum^{4}_{p = 0}\frac{\alpha_{p,r}(y_{k,l})}{p!} t^p,
        \quad r=1,2,
    \end{align}
    and then let $\hat x$ be the following normalization of $\bar x$:
    \begin{align}
        \hat x^{(r)}(t;y_{k,l}) = \frac{\bar x^{(r)}(t;y_{k,l}) - \min_{s} \bar x^{(r)}(s;y_{k,l})}{\max_{s}\bar x^{(r)}(s;y_{k,l}) - \min_{t}\bar x^{(r)}(s;y_{k,l})},
         \quad r=1,2.
    \end{align}
    
    \item To generate the components of $\bL$, we first compute
    \begin{align}
        I_{k,l} = \left\{ (i,j) | \exists s \in [0,1] \text{ s.t.}  (i,j) = \argmin_{(i,j)} \norm{x_{i,j} - \hat x(s;y_{k,l})} \right\}.
    \end{align}
    and then set the values of $\bL$ directly by
    \begin{align}
        L_{(k,l),(i,j)} = 
        \begin{cases}
        \frac{1}{|I_{k,l}|} &\text{ if }(i,j) \in I_{k,l}.\\
        0 &\text{ else }
        \end{cases}
    \end{align}
\end{enumerate}

\bibliographystyle{plain}
\bibliography{main.bbl}

\begin{thebibliography}{10}

\bibitem{arridge1999optical}
Simon~R Arridge.
\newblock Optical tomography in medical imaging.
\newblock {\em Inverse problems}, 15(2):R41, 1999.

\bibitem{Becker2012lbfgs}
Stephen Becker.
\newblock Lbfgsb (l-bfgs-b) mex wrapper, 2012--2015.

\bibitem{broyden1970convergence}
Charles~George Broyden.
\newblock The convergence of a class of double-rank minimization algorithms 1.
  general considerations.
\newblock {\em IMA Journal of Applied Mathematics}, 6(1):76--90, 1970.

\bibitem{chahine1970inverse}
Moustafa~T Chahine.
\newblock Inverse problems in radiative transfer: Determination of atmospheric
  parameters.
\newblock {\em Journal of the Atmospheric Sciences}, 27(6):960--967, 1970.

\bibitem{chan2005image}
Tony~F Chan and Jianhong~Jackie Shen.
\newblock {\em Image processing and analysis: variational, PDE, wavelet, and
  stochastic methods}, volume~94.
\newblock Siam, 2005.

\bibitem{daubechies1988orthonormal}
Ingrid Daubechies.
\newblock Orthonormal bases of compactly supported wavelets.
\newblock {\em Communications on pure and applied mathematics}, 41(7):909--996,
  1988.

\bibitem{engquist2013application}
Bjorn Engquist and Brittany~D Froese.
\newblock Application of the wasserstein metric to seismic signals.
\newblock {\em arXiv preprint arXiv:1311.4581}, 2013.

\bibitem{engquist2016optimal}
Bjorn Engquist, Brittany~D Froese, and Yunan Yang.
\newblock Optimal transport for seismic full waveform inversion.
\newblock {\em arXiv preprint arXiv:1602.01540}, 2016.

\bibitem{evans1997partial}
Lawrence~C Evans.
\newblock Partial differential equations and monge-kantorovich mass transfer.
\newblock {\em Current developments in mathematics}, 1997(1):65--126, 1997.

\bibitem{evans1999differential}
Lawrence~C Evans and Wilfrid Gangbo.
\newblock {\em Differential equations methods for the Monge-Kantorovich mass
  transfer problem}, volume 653.
\newblock American Mathematical Soc., 1999.

\bibitem{fletcher1970new}
Roger Fletcher.
\newblock A new approach to variable metric algorithms.
\newblock {\em The computer journal}, 13(3):317--322, 1970.

\bibitem{freeman1992sar}
Anthony Freeman.
\newblock Sar calibration: An overview.
\newblock {\em IEEE Transactions on Geoscience and Remote Sensing},
  30(6):1107--1121, 1992.

\bibitem{goldfarb1970family}
Donald Goldfarb.
\newblock A family of variable-metric methods derived by variational means.
\newblock {\em Mathematics of computation}, 24(109):23--26, 1970.

\bibitem{goldstein2009split}
Tom Goldstein and Stanley Osher.
\newblock The split bregman method for l1-regularized problems.
\newblock {\em SIAM journal on imaging sciences}, 2(2):323--343, 2009.

\bibitem{golub1996matrix}
Gene~H Golub.
\newblock {\em Matrix computations}.
\newblock Johns Hopkins University Press, 1996.

\bibitem{golub1999tikhonov}
Gene~H Golub, Per~Christian Hansen, and Dianne~P O'Leary.
\newblock Tikhonov regularization and total least squares.
\newblock {\em SIAM Journal on Matrix Analysis and Applications},
  21(1):185--194, 1999.

\bibitem{hansen1992analysis}
Per~Christian Hansen.
\newblock Analysis of discrete ill-posed problems by means of the l-curve.
\newblock {\em SIAM review}, 34(4):561--580, 1992.

\bibitem{hansen1993use}
Per~Christian Hansen and Dianne~Prost O’Leary.
\newblock The use of the l-curve in the regularization of discrete ill-posed
  problems.
\newblock {\em SIAM Journal on Scientific Computing}, 14(6):1487--1503, 1993.

\bibitem{kennedy2001bayesian}
Marc~C Kennedy and Anthony O'Hagan.
\newblock Bayesian calibration of computer models.
\newblock {\em Journal of the Royal Statistical Society: Series B (Statistical
  Methodology)}, 63(3):425--464, 2001.

\bibitem{kirsch2011introduction}
Andreas Kirsch.
\newblock {\em An introduction to the mathematical theory of inverse problems},
  volume 120.
\newblock Springer Science \& Business Media, 2011.

\bibitem{li2016fast}
Wuchen Li, Stanley Osher, and Wilfrid Gangbo.
\newblock A fast algorithm for earth mover's distance based on optimal
  transport and l1 type regularization.
\newblock {\em arXiv preprint arXiv:1609.07092}, 2016.

\bibitem{li2017parallel}
Wuchen Li, Ernest~K Ryu, Stanlet Osher, Wotao Yin, and Wolfred Gangbo.
\newblock A parallel method for earth mover’s distance.
\newblock {\em Journal of Scientific Computing}, page 75(1), 2018.

\bibitem{mallat1989multiresolution}
Stephane~G Mallat.
\newblock Multiresolution approximations and wavelet orthonormal bases of
  l$^2$(r).
\newblock {\em Transactions of the American mathematical society},
  315(1):69--87, 1989.

\bibitem{oliver2008inverse}
Dean~S Oliver, Albert~C Reynolds, and Ning Liu.
\newblock {\em Inverse theory for petroleum reservoir characterization and
  history matching}.
\newblock Cambridge University Press, 2008.

\bibitem{perlin1985image}
Ken Perlin.
\newblock An image synthesizer.
\newblock {\em ACM Siggraph Computer Graphics}, 19(3):287--296, 1985.

\bibitem{perlin2002improving}
Ken Perlin.
\newblock Improving noise.
\newblock In {\em ACM Transactions on Graphics (TOG)}, volume~21, pages
  681--682. ACM, 2002.

\bibitem{rudin1992nonlinear}
Leonid~I Rudin, Stanley Osher, and Emad Fatemi.
\newblock Nonlinear total variation based noise removal algorithms.
\newblock {\em Physica D: Nonlinear Phenomena}, 60(1-4):259--268, 1992.

\bibitem{ryu2018transport}
Ernest Ryu, Yongxin Chen, Wuchen Li, and Stanley Osher.
\newblock Vector and matrix optimal mass transport: Theory, algorithm, and
  applications.
\newblock {\em arXiv}, 2017.

\bibitem{schneider2012tomographic}
Kai Schneider, Romain Nguyen~van Yen, Nicolas Fedorczak, Frederic Brochard,
  Gerard Bonhomme, Marie Farge, and Pascale Monier-Garbet.
\newblock Tomographic reconstruction of tokamak plasma light emission using
  wavelet-vaguelette decomposition.
\newblock In {\em APS Meeting Abstracts}, 2012.

\bibitem{schneider1996calibration}
Uwe Schneider, Eros Pedroni, and Antony Lomax.
\newblock The calibration of ct hounsfield units for radiotherapy treatment
  planning.
\newblock {\em Physics in Medicine \& Biology}, 41(1):111, 1996.

\bibitem{shanno1970conditioning}
David~F Shanno.
\newblock Conditioning of quasi-newton methods for function minimization.
\newblock {\em Mathematics of computation}, 24(111):647--656, 1970.

\bibitem{villani2008optimal}
C{\'e}dric Villani.
\newblock {\em Optimal transport: old and new}, volume 338.
\newblock Springer Science \& Business Media, 2008.

\bibitem{wingen2015regularization}
Andreas Wingen, MW~Shafer, Ezekial~A Unterberg, Judith~C Hill, and Donald~L
  Hillis.
\newblock Regularization of soft-x-ray imaging in the diii-d tokamak.
\newblock {\em Journal of Computational Physics}, 289:83--95, 2015.

\bibitem{yang2018application}
Yunan Yang, Bj{\"o}rn Engquist, Junzhe Sun, and Brittany~F Hamfeldt.
\newblock Application of optimal transport and the quadratic wasserstein metric
  to full-waveform inversion.
\newblock {\em Geophysics}, 83(1):R43--R62, 2018.

\bibitem{zhu1994lbfgs}
Ciyou Zhu, Richard~H Byrd, Peihuang Lu, and Jorge Nocedal.
\newblock Lbfgs-b: Fortran subroutines for large-scale bound constrained
  optimization.
\newblock {\em Report NAM-11, EECS Department, Northwestern University}, 1994.

\end{thebibliography}
\end{document}